\documentclass[a4paper, 11pt]{amsart}
\usepackage[top=1.25in, bottom=1.25in, left=1in, right=1in]{geometry}
\usepackage{amsmath, amssymb,amsmath,amscd,amsfonts,amsthm,mathrsfs, graphicx}
\usepackage[arrow,matrix,curve,cmtip,ps]{xy}
\usepackage{enumitem}
\usepackage[alphabetic]{amsrefs}
\usepackage{pinlabel}
\usepackage{tikz-cd}
\usepackage{color}
\usepackage{xcolor}
\usepackage{caption}
\usepackage{subcaption}
\usepackage{soul}

\usepackage{cancel}

\usepackage{setspace}

\usepackage{relsize}

\usepackage{centernot}

\usepackage{hyperref}

\usepackage{pinlabel}	

\newtheorem{proposition}{Proposition}[section]
\newtheorem{theorem}[proposition]{Theorem}
\newtheorem{corollary}[proposition]{Corollary}
\newtheorem{lemma}[proposition]{Lemma}
\newtheorem{remark}[proposition]{Remark}

\newtheorem*{theorem*}{Theorem}
\newtheorem*{proposition*}{Proposition}
\newtheorem*{lemma*}{Lemma}
\newtheorem*{corollary*}{Corollary}
\newtheorem*{rep@theorem}{\rep@title}
\newcommand{\newreptheorem}[2]{
\newenvironment{rep#1}[1]{
 \def\rep@title{#2 \ref{##1}}
 \begin{rep@theorem}}
 {\end{rep@theorem}}}
\makeatother

\newtheorem*{rep@proposition}{\rep@title}
\newcommand{\newrepproposition}[2]{
\newenvironment{rep#1}[1]{
 \def\rep@title{#2 \ref{##1}}
 \begin{rep@proposition}}
 {\end{rep@proposition}}}
\makeatother

\theoremstyle{definition}
\newtheorem{definition}[proposition]{Definition}
\newtheorem{question}[proposition]{Question}

\newreptheorem{theorem}{Theorem}
\newreptheorem{lemma}{Lemma}
\newreptheorem{proposition}{Proposition}
\newreptheorem{corollary}{Corollary}

\newcommand{\bdry}{\partial}
\renewcommand{\Cup}{\bigcup}
\newcommand{\sm}{\smallsetminus}

\newcommand{\Q}{\mathbb{Q}}
\newcommand{\N}{\mathbb{N}}
\newcommand{\Z}{\mathbb{Z}}
\newcommand{\F}{\mathcal{F}}
\newcommand{\B}{\mathcal{B}}
\newcommand{\C}{\mathcal{C}}

\newcommand{\A}{\mathcal{A}}
\newcommand{\M}{\mathcal{M}}
\newcommand{\im}{\operatorname{Im}}
\newcommand{\lk}{\operatorname{lk}}

\newcommand{\Id}{\operatorname{Id}}
\newcommand{\Arf}{\operatorname{Arf}}

\newcommand{\Bl}{\mathcal{B}\ell}

\newcommand{\double}{\operatorname{Z}}

\newcommand{\interior}{\operatorname{int}}

\newcommand{\into}{\hookrightarrow}

\newcommand{\immerse}{\looparrowright}

\newcommand{\pref}[1]{(\ref{#1})}

\newcommand{\Oplus}{\bigoplus}
\begin{document}
\title[Concordance to boundary links and the solvable filtration]{Concordance to boundary links and the solvable filtration}

\author{Christopher W.\ Davis}
\address{Department of Mathematics, University of Wisconsin--Eau Claire}
\email{daviscw@uwec.edu}

\author{Shelly Harvey $^{\dag}$}
\address{Department of Mathematics, Rice University}
\email{shelly@rice.edu}

\author{JungHwan Park}
\address{Department of Mathematical Sciences, Korea Advanced Institute for Science and Technology}
\email{jungpark0817@kaist.ac.kr}

\thanks{$^{\dag}$ This work was partially supported National Science Foundation grants DMS-2109308 and DMS-1613279.}

\date{\today}

\def\subjclassname{\textup{2020} Mathematics Subject Classification}
\expandafter\let\csname subjclassname@1991\endcsname=\subjclassname

\begin{abstract} 
A geometric interpretation of the vanishing of Milnor's higher order linking numbers remains an important open problem in the study of link concordance.  In the 1990's, works of Cochran-Orr and Livingston exhibit a potential resolution to this problem in the form of homology boundary links.  They exhibit the first known links with vanishing Milnor's invariants that are not concordant to boundary links.  It remains unknown whether every link with vanishing Milnor's invariants is concordant to a homology boundary link.  In this paper we present an obstruction to concordance to a homology boundary link and a potential path to the construction of links with vanishing Milnor's invariants which are not concordant to a homology boundary link.  Our obstructions and examples fit into the language of the solvable filtration due to Cochran-Orr-Teichner.  Along the way we demonstrate that every homology boundary link is equivalent to a boundary link modulo the 0.5 term of the solvable filtration, in contrast to the results of Cochran-Orr and Livingston.  Finally we exhibit highly solvable homology boundary links that are not concordant to boundary links, pushing the work of Cochran-Orr and Livingston deep into the solvable filtration.

\end{abstract}

\maketitle
\section{Introduction}
Knots and links in $S^3$ play an essential role in the classification of 3- and 4-dimensional manifolds.  While $3$-dimensional manifolds are well understood using Thurston's Geometrization program, $4$-dimensions (in both the smooth and topological categories) is still elusive. Arguably, the most important equivalence relation for the study of $4$-dimensional manifolds is knot and link concordance.  For example, if $K$ is a knot in $S^3$ that is topologically (locally flat) slice but not smoothly slice then one can use the trace of $K$ to create an exotic $\mathbb{R}^4$, a smooth $4$-manifold that is homeomorphic but not diffeomorphic to $\mathbb{R}^4$ with the usual smooth structure (see Exercise $9.4.23$ and its solution on $p. 522$ in \cite{KirbyCalculus}).  The existence of such an exotic structure was pointed out by Freedman in 1982, using his work on topological $4$-manifolds \cite{Freedman1} and Donaldson's theorem on the intersection forms of smooth $4$-manifolds \cite{Donaldson}.  Strangely, this is the only dimension in which this happens; that is, if $\mathbb{R}^n$ has an exotic structure then $n=4$.  

While the set of knots up to concordance has seen much progress (both topologically and smoothly) over the last 10 years, studying links up to concordance proves to be much more difficult. Because of this, many others have studied the 
category of links most closely resembling knots, namely boundary links. A link $L$ in $S^3$ is called a \emph{boundary link} if the components of $L$ bound disjoint Seifert surfaces.  Given such a collection of disjoint Seifert surfaces one can define invariants via a Seifert matrix as one does for a knot. See for example~\cite{KiHyoungKo1, KiHyoungKo2,Sheiham03}.
However, because of the elusive boundary concordance versus link concordance for boundary links,  studying even these classical  invariants is difficult.  The main reason is that even if a boundary link is slice (smoothly or topologically), it is not known whether its Seifert form must admit a strong algebraic metabolizer. However, if your link is boundary slice (a conjecturally much more restrictive condition), then its Seifert form admits a strong metabolizer.  As a consequence, many invariants from the Seifert matrix for $L$ are boundary concordance invariants.  

These tools give a potential avenue to understanding boundary links up to concordance.  It is natural to now ask if every concordance class of a link is represented by a boundary link.

The first obstructions to a link being concordant to a boundary link are Milnor's invariants.  For any $L$ with $m$ components, there is a map from the free group of rank $m$, $F$ to $\pi_1(S^3 \smallsetminus L)$ sending the generators to the meridians.  When $L$ is a boundary link, there is a surjective map to $F$ that is a retract modulo the terms of the lower central series.  The latter map sends the longitudes of $L$ to the trivial element in $F$ so $L$ has vanishing Milnor's invariants.  
As Milnor's invariants are preserved by concordance~\cite{Ca}, any link concordant to a boundary link has vanishing Milnor's invariants.  In groundbreaking works, Cochran-Orr~\cite{Cochran-Orr:1990-1, Cochran-Orr:1993-1}, and subsequently Livingston~\cite{Livingston:1990-1}, demonstrate the existence of links which are not concordant to boundary links even though their Milnor's invariants vanish.  These links sit in a larger class of links called homology boundary links.  It is not known if every link with vanishing Milnor's invariants is concordant to a homology boundary link.  It is also unknown whether a sublink of a homology boundary link must be concordant a homology boundary link. Summarizing the current state of the art, 
\begin{equation}\label{eqn: comparison}
\mathcal{BL}\subsetneq \mathcal{HBL} \subseteq \mathcal{SHBL}\subseteq \mathcal{M}.
\end{equation}
The symbol $\subseteq$ is used to emphasize that the reverse containment is unknown, $\mathcal{BL}$ is the set of links concordant to boundary links, $\mathcal{HBL}$ is the set of links concordant to homology boundary links, $\mathcal{SHBL}$ is the set of links concordant to sublinks of homology boundary links, and $\mathcal{M}$ is the set of links with vanishing Milnor's invariants.

Let $\mathcal{C}$ be the set of links up to concordance. In \cite{COT03}, Cochran-Orr-Teichner define a filtration of $\mathcal{C}$ by subsets, called the \emph{solvable filtration}, denoted as follows 
$$\dots \subseteq \F_{n+1}\subseteq \F_{n.5}\subseteq \F_{n}\subseteq \dots \subseteq \F_{0.5}\subseteq \F_0\subseteq \C.$$ Links in $\F_n$ are called $n$-solvable. In \cite{Davis:2020-1}, the first named author defines a sequence of increasingly strong equivalence relations on $\C$ called {$n$-solve-equivalence} denoted $\simeq_n$ which extends the solvable filtration in that the $n$-solve-equivalence class of the $m$-component unlink is precisely the set of $m$-component $n$-solvable links.  

As a consequence of Martin's classification of links up to 0-solvability \cite[Theorem 1]{Martin22} if a link has vanishing Milnor's invariants of length up to 4 then that link is $0$-solve-equivalent to a boundary link.  Otto \cite[Theorem 3.1]{Otto14} proves that if a link is in $\F_n$ with $n\in \N$ then Milnor's invariants of length up to $2^{n+2}-1$ vanish.  While we do not do so here, her proof generalizes to show that the same conclusion follows if a link is $n$-solve-equivalent to a sublink of a homology boundary link.  Otto goes on to show that that there exist $n$-solvable links with nonvanishing Milnor's invariants of length $2^{n+2}$, so that Milnor's invariants of length $2^{n+2}$ are not invariants of $n$-solve-equivalence.   This suggests that $n$-solve-equivalence to a sublink of a homology boundary link might be detected by a finite collection of Milnor's invariants.  We prove that this is not the case, even for $n=0.5$. 
  
 \begin{theorem}\label{thm: n=0}
For each positive integer $k$, there exists a 0-solvable link $L$ with slice components such that 
\begin{enumerate}[font=\upshape]
\item\label{item:vanishingMuFor0.5} $\mu_I(L) = 0$ for each multi-index $I$ with $|I|< k$, and 
\item $L$ is not $0.5$-solve-equivalent to any sublink of a homology boundary link.
\end{enumerate}\end{theorem}

This serves as evidence that $\mathcal{SHBL}\subsetneq \mathcal{M}$.  Indeed, if one could improve conclusion \pref{item:vanishingMuFor0.5} of Theorem~\ref{thm: n=0} to say that all Milnor's invariants of $L$ vanish then one will have proven that the containment is proper. In order to motivate a strategy for such an improvement, we take a moment and summarize the proof of Theorem~\ref{thm: n=0}.  The tool we use to obstruct concordance to boundary links is the Blanchfield form, which we discuss now.  Let $L=L_1\cup L_2$ be a link with linking number $0$ and $\widetilde{L_2}$ be a lift of $L_2$ to the infinite cyclic cover of the exterior $L_1$.  The (rational) first homology of this infinite cyclic cover is the \emph{(rational) Alexander module} of $L_1$, denoted by $\A(L_1)$. The homology class of $\widetilde{L_2}$ is denoted by $[\widetilde{L_2}]\in \A(L_1)$. There is a sesquilinear form $$\Bl_{L_1}\colon\A(L_1)\times \A(L_1) \to \Q(t)/ \Lambda,$$ called the \emph{Blanchfield form}. Here and throughout the article we write $\Lambda := \mathbb{Q}[t,t^{-1}]$ for the Laurent polynomial ring.  In \cite{Davis-Park-2021} the first and third authors prove that this form obstructs concordance to links with an unknotted component.  As we shall prove, it also obstructs $L$ being concordant to a homology boundary link.

\begin{repproposition}{prop: Blanchfield obstructs HBL} If $L=L_1\cup L_2$ is $0.5$-solve-equivalent to a sublink of a homology boundary link, then $L_2$ lifts to a knot $\widetilde{L_2}$ in the infinite cyclic cover of the exterior of $L_1$.  Let $[\widetilde{L_2}]\in \A(L_1)$ be its rational homology class, then $\Bl_{L_1}([\widetilde{L_2}],[\widetilde{L_2}])=0$.  
\end{repproposition}

Let $L=L_1\cup L_2$ be any link with linking number zero and $\Bl_{L_1}([\widetilde{L_2}],[\widetilde{L_2}])\neq 0$.  Let $k_0$ be a positive integer such that $\mu_I(L)=0$ for each multi-index $I$ with length $|I| < k_0$.  An example with $k_0=4$ appears in Figure~~\ref{fig: 0.5 example}.  By work of Cochran \cite[Theorem 3.3]{Cochran91}, there is another link $J$ which has unknotted components and which satisfies $\mu_I(L)=\mu_I(J)$ for all multi-indices with $|I| \le k_0$.
By the additivity of the first non-vanishing Milnor's invariants \cite[Theorem 8.13]{Cochran91}, we see that the link $L' = L_1'\cup L_2' = L\#_\delta -J$ obtained by an exterior band sum of $L$ with the reverse of the mirror image of $J$ has vanishing Milnor's invariants of length less than $k_0+1$. Basic properties of the Blanchfield form imply that $\Bl_{L_1'}([\widetilde{L'_2}],[\widetilde{L'_2}]) = \Bl_{L_1}([\widetilde{L_2}],[\widetilde{L_2}])\neq 0$.  Combined with Proposition~\ref{prop: Blanchfield obstructs HBL}, this implies that $L'$ is not $0.5$-solve-equivalent to a sublink of homology boundary link.  Replace $L$ by $L'$ and iterate the argument above to prove Theorem~\ref{thm: n=0}.
\vspace{1cm}
  \begin{figure}[h]
\begin{subfigure}{.3\textwidth}
         \centering
\begin{picture}(80,110)
\put(0,0){\includegraphics[height=.15\textheight]{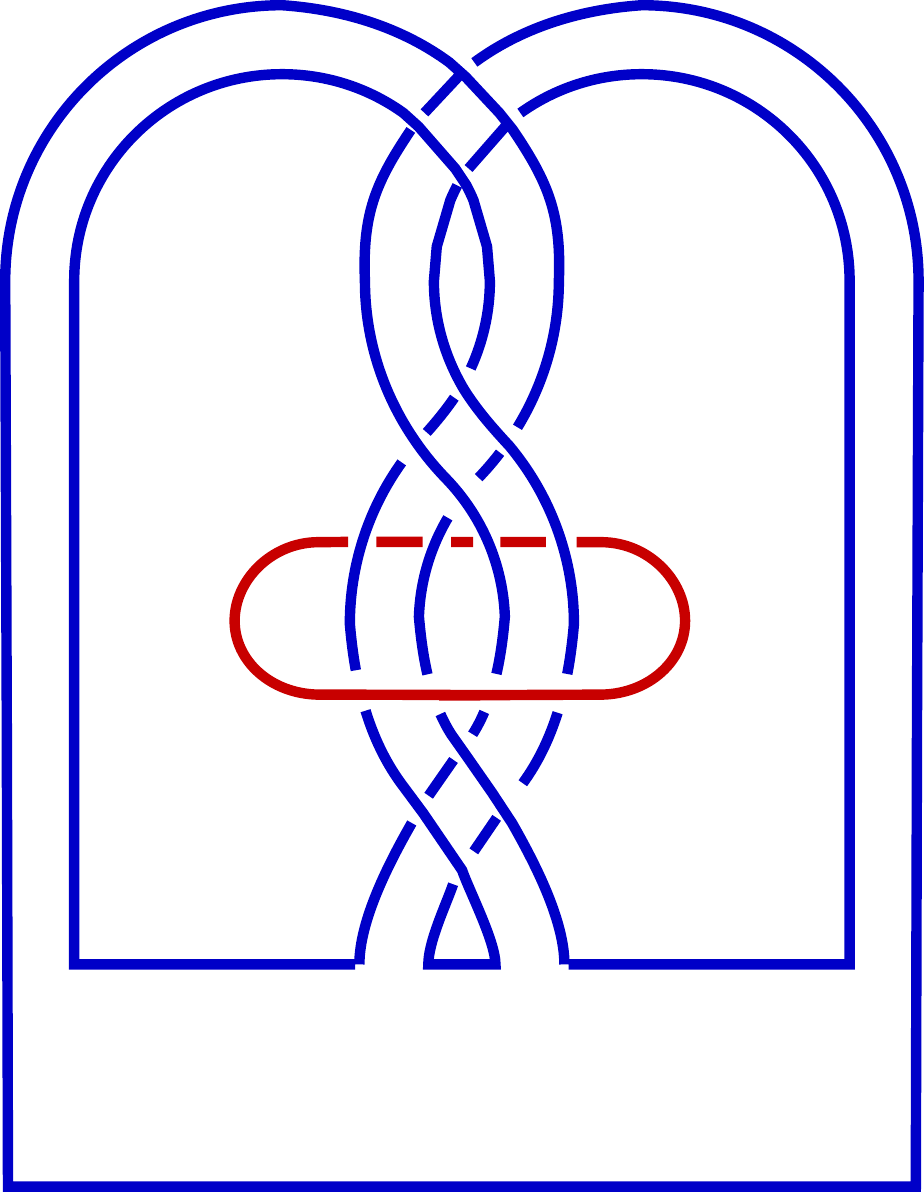}}
\put(34,5) {$L_1$}
\put(10,37) {$L_2$}
\end{picture}
\caption{}\label{fig: 0.5 example}
\end{subfigure}
\begin{subfigure}{.3\textwidth}         
\centering
\begin{picture}(95,110)
\put(-2,0){\includegraphics[height=.15\textheight]{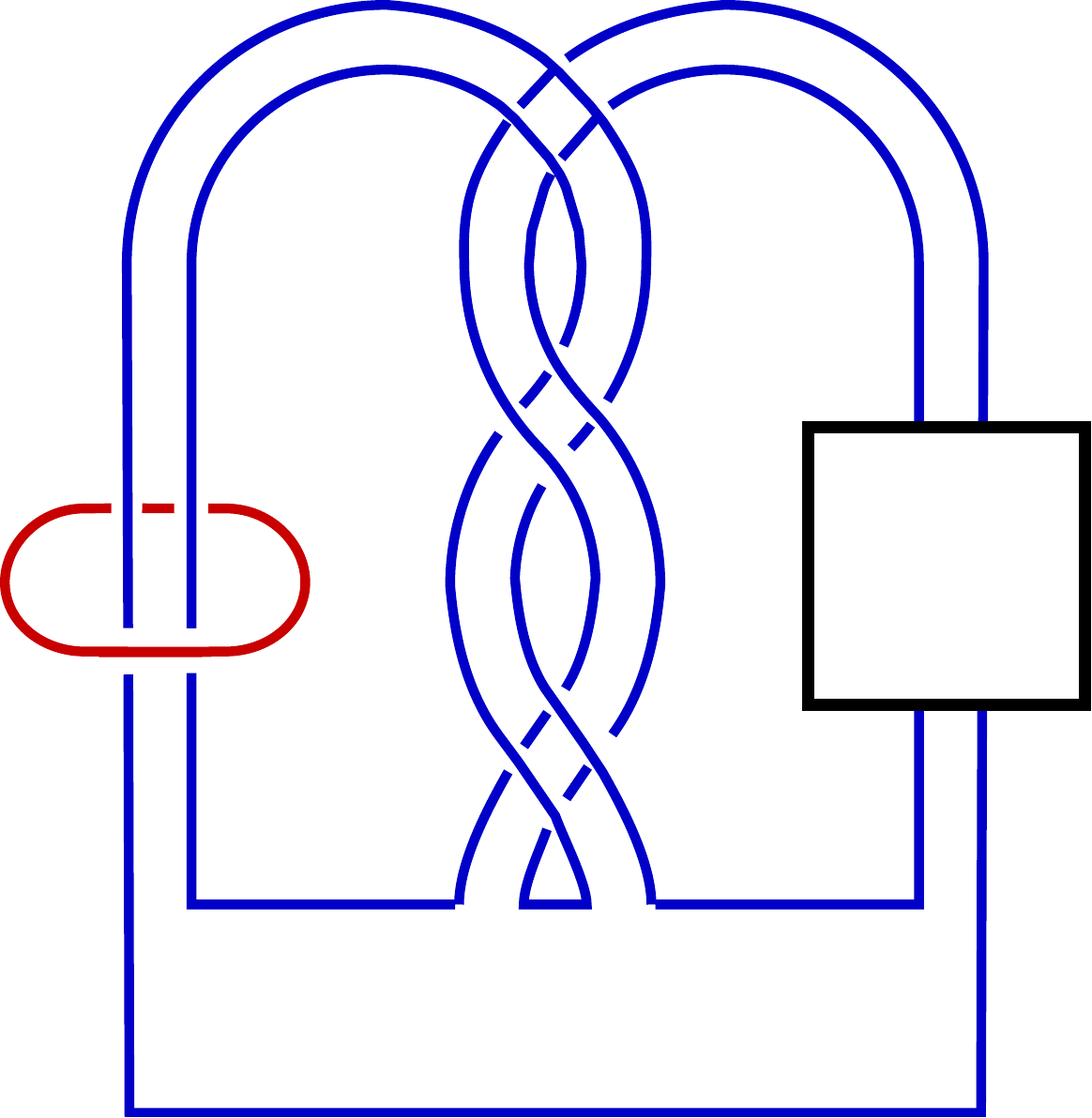}}
\put(78,45){$J$}
\put(42,5) {$L_1$}
\put(-11,37) {$L_2$}
\end{picture}
\caption{}\label{fig: highly sol example}
\end{subfigure}
\caption{Left:  A link $L$ satisfying the assumptions of Proposition~\ref{prop: Blanchfield obstructs HBL}.  Right:  For any natural number and a suitable choice of $J$ this link is $n$-solvable and yet is not $n.5$-solve-equivalent to a boundary link. }
\end{figure}
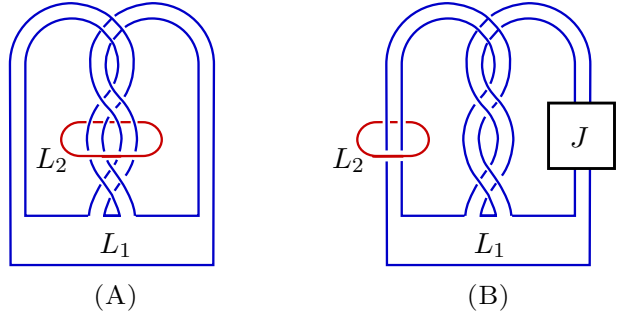 
\vspace{0.2cm}

Recall that if $|I|<|J|$ and $\mu_I(L)\neq 0$, then $\mu_J(L)$ is not well defined.  As a consequence, we cannot extend \cite[Theorem 3.3]{Cochran91} to Milnor's invariants of length greater than $k+1$ directly. On the other hand, if $S$ is a string link then $\mu_I(S)$ is well defined for every multi-index $I$, regardless of the vanishing of lower order Milnor's invariants~\cite{HL2}.
  (In fact \cite{HL2} works with the Artin representation, whose vanishing is equivalent to the vanishing of Milnor's invariants.  In a slight abuse of notation we will still refer to this in the language of Milnor's invariants of~$S$.) The same arguments as referenced in the previous paragraph together with induction reveal that for any string link $S$ and $k\in \N$ there is a string link $T$ with unknotted components with $\mu_I(S)=\mu_I(T)$ for all $|I|\le k$.  It natural to ask whether the condition that $|I|\le k$ can be removed. 

\begin{question}\label{quest: unknotted Milnor}
Let $S$ be a string link.  Does there exist a string link $T$ with unknotted components which satisfies that $\mu_I(S)=\mu_I(T)$ for every multi-index $I$?  This is equivalent to asking if the closure, $\widehat{S*T^{-1}}$ has vanishing Milnor's invariants.  (As a reminder $T^{-1}$, the inverse of $T$ in the string link concordance group, is the result of reflecting  $T\subseteq D^2\times[0,1]$ across $D^2\times\{1/2\}$  and then reversing all of the arrows.)   
\end{question}

Currently we do not know how to resolve this question.  However, no finite collection of {Milnor's} invariants can be used to deduce that a link has a non-trivially knotted component \cite[Theorem 3.3]{Cochran91}.
 It would be surprising if the collection of all Milnor's invariants were able to do so. Assuming this question is answered in the affirmative, we prove that the containment $\mathcal{SHBL}\subseteq \mathcal{M}$  of \pref{eqn: comparison} is proper.

\begin{theorem}\label{thm: n=0 vsn 2}
If the answer to Question~\ref{quest: unknotted Milnor} is affirmative, then there exists a 0-solvable link $L$ with slice components such that 
\begin{enumerate}[font=\upshape]
\item $L$ has vanishing Milnor's invariants.
\item $L$ is not $0.5$-solve-equivalent to any sublink of a homology boundary link.
\end{enumerate}
In particular, $L$ is not {concordant} to any sublink of a homology boundary link even though it has vanishing Milnor's invariants. 
\end{theorem}

\begin{proof}Let $L^0=L_1^0\cup L_2^0$ be a link 
with {$\lk(L_1^0, L_2^0)=0$ and $\Bl_{L_1^0}([\widetilde{L_2^0}],[\widetilde{L_2^0}])\neq0$}, so that by Proposition~\ref{prop: Blanchfield obstructs HBL} $L$ is not 0.5-solvably equivalent to a sublink of a homology boundary link.
  Pick a string link $S$ whose closure is $L^0$ and for which there is another string link $T$ with unknotted components so that $L=L_1\cup L_2:=\widehat{S*T^{-1}}$ has vanishing Milnor's invariants.  The vanishing of the Milnor's invariants of $S*T^{-1}$ implies the vanishing of the Milnor's invariants of $L$ \cite[Lemma 3.7]{HL2}.

The unknottedness of the components of $T$ implies that $L_1$ is isotopic to $L_1^0$.  If we identify $L_1$ and $L_1^0$, then $L_2$ is homotopic to $L_2^0$ in the exterior of $L_1$.  Then we have $\Bl_{L_1}([\widetilde L_2],[\widetilde L_2]) = \Bl_{\widetilde L_1^0}([\widetilde L_2^0],[\widetilde L_2^0])\neq0$.  Thus, by Proposition~\ref{prop: Blanchfield obstructs HBL}, $L$ is not $0.5$-solve-equivalent to a sublink of a homology boundary link.  
\end{proof}

Recall that the main theorem of \cite{Cochran-Orr:1990-1, Cochran-Orr:1993-1, Livingston:1990-1} is the existence of homology boundary links which are not concordant to boundary links. Cochran-Orr and Livingston use the linking form of branched covers and Casson-Gordon-invariants, respectively. 
While the precise interaction of these invariants with the solvable filtration appears to be unknown for links, Casson-Gordon invariants of $1.5$-solvable knots vanish \cite[Theorem 9.11]{COT03} and so the tools of Cochran-Orr and Livingston seem likely to fail to detect highly solvable links. 
  Links that sit deep in the solvable filtration can be thought of as  very close to being slice.
 The second goal of this paper is to find homology boundary links which lie deep in the solvable filtration and still are not concordant to a boundary link. More precisely, we prove the following.

\begin{theorem}\label{thm: main}
For each positive integer $n$, there exists an $n$-solvable homology boundary link which is not $n.5$-solve-equivalent to a boundary link.\end{theorem}

For the sake of being explicit, the link of Theorem~\ref{thm: main} appears in Figure~\ref{fig: highly sol example} for an appropriate choice of the knot $J$. Note that in the above theorem we require $n$ to be a positive integer. There is no distinction between boundary links and homology boundary links up to $0$-solve-equivalence. Indeed, as a consequence of \cite[Theorem 1]{Martin22} any homology boundary link, and more generally any link with vanishing Milnor's invariants of length up to 4 is $0$-solve-equivalent to a boundary link. We extend this further and prove that homology boundary links are equivalent to boundary links up to $0.5$-solve-equivalence also.

\begin{theorem}\label{thm: HBL = BL}
For any link $L$, the following are equivalent:
\begin{enumerate}[font=\upshape]
\item $L$ is $0.5$-solve-equivalent to a boundary link;
\item $L$ is $0.5$-solve-equivalent to a homology boundary link;
\item $L$ is $0.5$-solve-equivalent to a sublink of a homology boundary link.
\end{enumerate}
\end{theorem}
 Summarizing, after modding out by $0.5$-solve-equivalence, 
$$(\mathcal{BL} /\simeq_{0.5})=(\mathcal{HBL}/\simeq_{0.5}) =(\mathcal{SHBL}/\simeq_{0.5}) \subseteq (\M/\simeq_{0.5}).$$
According the Theorem~\ref{thm: n=0 vsn 2}, if the answer to Question~\ref{quest: unknotted Milnor} is affirmative, then $(\mathcal{SHBL}/\simeq_{0.5}) \subsetneq (\M/\simeq_{0.5}).$

\subsection*{Generalized Seifert matrices and $0.5$-solve-equivalence.}

Recall that $0.5$-solvability of knots is detected by their Seifert form \cite[Theorem 1.1]{COT03}.  In order to prove Theorem~\ref{thm: HBL = BL} we generalize this and generate a new sufficient condition for links to be $0.5$-solve-equivalent.

Up to concordance, the set of homology boundary links is the same as the set of fusions of boundary links \cite{CocLev91}.  We recall the precise definition of fusion in Section~\ref{sect: HBL = BL for 0.5} .  These links bound a special type of immersed surface which we call fusion complexes.  These surfaces have Seifert forms defined analogously to the Seifert form of a boundary link.  In Section~\ref{sect: HBL = BL for 0.5}, we define what it means for these forms to be isomorphic, and show that if these forms are isomorphic then the underlying links are $0.5$-solve-equivalent.

\begin{reptheorem}{thm: S-equiv implies 0.5-sol}
If $L$ and $L'$ are fusions of boundary links with fusion surfaces $G$ and $G'$ that have isomorphic Seifert forms, then $L$ and $L'$ are $0.5$-solve-equivalent. \end{reptheorem}

In Lemma~\ref{lem: S-equiv to bdry} we find that every fusion of boundary links shares its Seifert form with some boundary link.  Theorem~\ref{thm: HBL = BL} follows from these two results.

\subsection*{Notation and conventions} 
For any link $L$, the $0$-framed surgery on $S^3$ along $L$ is denoted by $M_L$. If $X$ is a submanifold properly embedded in a compact 4-manifold $W$, then an open tubular neighbourhood is denoted by $\nu X$. 

\subsection*{Organization} In Section~\ref{sect: solvable filtration}, we define solve-equivalence and explore its interaction with the solvable filtration and the operation of band sums.  In Section~\ref{sect: non 0.5 equiv to bdry}, we review the definition and properties of the Blanchfield form needed in the proof of Proposition~\ref{prop: Blanchfield obstructs HBL} and close with the proof of Theorem~\ref{thm: n=0}.  Section~\ref{sect: C-cplx}  consists of a generalization of a Seifert surface known as a C-complex along with Theorem~\ref{thm: metabolic C-cplx implies 0.5-sol}, a condition analogous to algebraic sliceness of knots which will imply 0.5-solvability of a link.  Theorem~\ref{thm: metabolic C-cplx implies 0.5-sol} is used in Section~\ref{sect: HBL = BL for 0.5} together with the notion of a fusion of boundary links \cite{CocLev91} to prove Theorem~\ref{thm: HBL = BL}.  In Section~\ref{sect:High sol background}, we begin thinking about links deep in the solvable filtration, and prove Theorem~\ref{thm: main} in Section~\ref{sect: HBL =/= BL for n.5}.  
\subsection*{Acknowledgements}
We are grateful to Jae Choon Cha and Kent Orr for useful conversations.

\section{The solvable filtration and solve-equivalence}\label{sect: solvable filtration}

In this section, we recall the solvable filtration of~\cite{COT03} and motivate  the definition of solve-equivalence of links appearing in~\cite{Davis:2020-1}. These filtrations are indexed by non-negative half-integers.  As such we will set $\N_0 := \{0,1, 2, \dots\}$ and $\frac{1}{2}\N_0 := \{0,0.5, 1, 1.5, \dots\}$.

 \begin{definition}
[\cite{COT03,CHL2011,Otto14}]\label{defn:solvable}
Let $L$ be an $m$-component link.
  We say that $L$ is \emph{$n$-solvable} ($n\in\N_0$) if there exists a smooth compact oriented 4-manifold  $W$ called an \emph{$n$-solution} bounded by $M_L$ such that:

   \begin{enumerate}[font=\upshape]

   \item $H_1(W)\cong \Z^m$ is generated by the meridians of $L$;
   
   \item \label{solvable classes}  For some $k\in \N_0$, there is a basis for $H_2(W)$ which is represented by $2k$ smoothly embedded, connected, compact, oriented surfaces $\{X_i, Y_i\}_{1\leq i \leq k}$ in $W$ with trivial normal bundles  all disjoint except that $X_i$ intersects $Y_i$ transversely in a single point for each $i$;
   
 \item \label{lift basis} For each $i$, the images of $\pi_1(X_i)\to \pi_1(W)$ and $\pi_1(Y_i)\to \pi_1(W)$ are both contained in $\pi_1(W)^{(n)}$.  
   \end{enumerate}

The surfaces $X_i$ and $Y_i$ are called \emph{$n$-Lagrangians} and \emph{$n$-duals}, respectively.  We say that 
$L$ is \emph{$(n.5)$-solvable} and  $W$ is an \emph{$(n.5)$-solution} if additionally the image of  $\pi_1(X_i)\to \pi_1(W)$ is contained in $\pi_1(W)^{(n+1)}$ for each $i$.  \end{definition}
   
Let $\F_n$ (respectively $\F_{n.5}$) denote the set of all $n$-solvable (respectively $n.5$-solvable) links up to concordance.

   \begin{remark}
  The derived series of a group $G$ is given by $G^{(0)} = G$ and $G^{(n+1)} = [G^{(n)}, G^{(n)}]$.  In particular, when $n=0$ there is no content to condition~\eqref{lift basis} of Definition~\ref{defn:solvable}.  
   \end{remark}
   
   If a link is $n$-solvable with $n$-solution $W$, then one can construct a new 4-manifold $W'$ with boundary $S^3$ by attaching 2-handles to $W$ along the meridians of $L$.  The components of $L$ bound disjoint smooth disks in this 4-manifold.  More specifically, they bound the cocores of the added 2-handles.  By tracking some conditions about this 4-manifold and these disks one gets an equivalent formulation of solvablility in terms of bounded disks, which we state formally now.
   
   \begin{proposition}\label{prop: solvability in terms of disks}
  An $m$-component link $L$ is $n$-solvable if and only if there exists a smooth 4-manifold $W$ bounded by $S^3$ in which $L$ bounds a disjoint union of smoothly embedded disks $\Delta$ such that
\begin{enumerate}[font=\upshape]
   \item $H_1(W)=0$;
   \item \label{prop solvable classes} For some $k\in \N_0$, there is a basis for $H_2(W)$ which is represented by $2k$ smoothly embedded, connected, compact, oriented surfaces $\{X_i, Y_i\}_{1\leq i \leq k}$ in $W \sm \nu\Delta$ with trivial normal bundles  all disjoint except that $X_i$ intersects $Y_i$ transversely in a single point for each $i$;
 \item \label{prop lift basis} For each $i$, the images of $\pi_1(X_i)\to \pi_1(W\sm \nu\Delta)$ and $\pi_1(Y_i)\to \pi_1(W\sm\nu\Delta)$ are both contained in $\pi_1(W\sm\nu\Delta)^{(n)}$.  
   \end{enumerate}
Moreover, $L$ is $(n.5)$-solvable if and only if additionally the image of $\pi_1(X_i)\to \pi_1(W\sm\nu\Delta)$ is contained in $\pi_1(W\sm \nu\Delta)^{(n+1)}$ for each $i$. \end{proposition}

\begin{proof}This fact is well known to experts, and involves no hard ideas.  Thus, we merely summarize the argument.  Checking these details is a useful exercise for a reader attempting to familiarize themselves with the definition of solvability.   If $W$ and $\Delta$ are as in the statement of the proposition, then $W\sm\nu\Delta$ is an $n$-solution.  Conversely, if $W$ is an $n$-solution, then add 2-handles to $W$ along the meridians of $L$.  The resulting 4-manifold $W'$ is bounded by $S^3$ and the components of $L$ bound the cocores of the attached 2-handles in $W'$. If $\Delta$ is  the union of these cocores, then since $W=W'\smallsetminus \nu\Delta$, the conclusion follows.\end{proof}
   
Instead of looking for disks in 4-manifolds bounded by $S^3$, look for annuli in cobordisms from $S^3$ to itself.  Doing so inspires a filtered version of the equivalence relation of concordance.   This definition appears in \cite{Davis:2020-1} and  we recall it now.  Given a link $L$, the link obtained by reversing the orientation of every component of the mirror image of $L$ is denoted by $-L$.  
   
\begin{definition}[{\cite[Definition 2.3]{Davis:2020-1}}]\label{defn:n-solvableEquiv}
Let $L= L_1\cup\dots\cup L_m$ and $L'= L'_1\cup\dots\cup L'_m$ be $m$-component links.  We say that $L$ is \emph{$n$-solve-equivalent} to $L'$ and write $L\simeq_{n} L'$
 if there exists a smooth cobordism $W$ from $\bdry_+ W \cong S^3$ to $\bdry_- W \cong S^3$ with $L\subseteq \bdry_+ W$ and $L'\subseteq \bdry_- W$ together with smooth properly embedded annuli $C = C_1\cup\dots\cup C_m$ in $W$ with $\bdry C_i$ equal to $L_i\cup -L'_i$ such that:
\begin{enumerate}
\item $H_1(W)=0$;

\item \label{equiv classes} 
For some $k\in \mathbb{N}_0$, there is a basis for $H_2(W)$ which is represented by $2k$ smoothly embedded, connected, compact, oriented surfaces $\{X_i, Y_i\}_{1\leq i \leq k}$ in $W \sm \nu C$ with trivial normal bundles  all disjoint except that $X_i$ intersects $Y_i$ transversely in a single point for each $i$;

\item \label{equiv lift} For each $i$, the images of $\pi_1(X_i)\to \pi_1(W \sm \nu C)$ and $\pi_1(Y_i)\to \pi_1(W \sm \nu C)$ are both contained in $\pi_1(W \sm \nu C)^{(n)}$. 
\end{enumerate}

The pair $(W,C)$ is called an \emph{$n$-solve-equivalence}, and the surfaces $\{X_i\}_{1\leq i \leq k}$ and $\{Y_i\}_{1\leq i \leq k}$ are called \emph{$n$-Lagrangians} and \emph{$n$-duals}, respectively.  We say that $L$ is \emph{$n.5$-solve-equivalent} to $L'$ and $(W,C)$ an \emph{$(n.5)$-solve-equivalence} if additionally the image of $\pi_1(X_i)\to \pi_1(W \sm \nu C)$ is contained in $\pi_1(W \sm \nu C)^{(n+1)}$ for each $i$.
\end{definition}

This definition is compatible with the notion of solvabilty in that for any $n\in \frac{1}{2}\N_0$, the $n$-solve-equivalence class of the unlink is precisely the set of $n$-solvable links \cite[Proposition 2.4]{Davis:2020-1}. 

There is a generalization of the connected sum operation from the setting of knots to links.  Let $L=L_1\cup\dots \cup L_n$ and $L'=L_1'\cup\dots \cup L_n'$ be links in $S^3$ and $S\subseteq S^3$ be a 2-sphere with $L$ and $L'$ sitting in sidderent components of $S^3\setminus S$.
 Let $I=[0,1]$ and $\delta_1,\dots, \delta_n\colon I\times I \into S^3$ be a disjoint collection of bands satisfying that:
\begin{itemize}
\item $\delta_i\left[I\times\{0\}\right]$ is an arc in $L_i$, $\delta_i[I\times\{1\}]$ in an arc in $L_i'$ with the opposite orientation, and $\delta_i[I\times I]$ is otherwise disjoint from $L\cup L'$.
\item $\delta_i$ intersects $S$ in transversely $\delta_i[I\times\{1/2\}]$.  
\end{itemize}
The \emph{exterior band sum} of $L$ and $L'$ along $\delta$, denoted $L\#_\delta L'$ is given by starting with $L\sqcup L'$, removing $\delta_i[I\times\{0,1\}]$ and replacing it with $\delta_i[\{0,1\}\times I]$ for $i=1,\dots, n$.  Unlike knot concordance, band summing is not well defined on concordance.  The choice of bands can change the concordance class of the result.  However, it is not hard to see that if $L\#_\delta -L'$ is slice then $L$ and $L'$ are concordant.  The same proof applies to solvability.

\begin{proposition}\label{prop: solvability and interior sum}
Let $L$ and $L'$ be links and $n\in \frac{1}{2}\N$. If an exterior band sum $L\#_\delta -L'$ is $n$-solvable, then $L$ is $n$-solve-equivalent to $L'$.
\end{proposition}

\begin{proof}
Let $X, Y,$ and $Z$ be copies of $S^3$ containing $L, -L',$ and  $L\#_\delta -L'$, respectively.  Let $V$ be the 4-dimensional cobordism from $X\sqcup Y$ to $Z$ obtained by adding a single 1-handle to $X\times[0,1] \sqcup Y\times[0,1]$.  Construct a disjoint union of smoothly embedded pairs of pants $P$ in $V$ whose components are bounded by $L$, $-L'$, and $L\#_\delta -L'$ by starting with the disjoint union of annuli $L\times[0,1] \sqcup -L'\times[0,1]$ and adding 2-dimensional 1-handles along $\delta$ in $Z$.  

Since $L\#_\delta -L'$ is $n$-solvable, we appeal to Proposition~\ref{prop: solvability in terms of disks} to produce a 4-manifold $W$ in which $L\#_\delta -L'$ bounds a disjoint union of embedded disks $\Delta$. Observe that $P\cup \Delta\subseteq V\cup W$ is  a disjoint union of embedded annuli. We leave it as a straightforward exercise to the readers to verify that $\left(V\cup W, P\cup \Delta\right)$ is an $n$-solve-equivalence  from $L$ to $L'$.  
\end{proof}

\section{The Blanchfield form and $0.5$-solvability.}\label{sect: non 0.5 equiv to bdry}

We begin by recalling some details concerning the Alexander module, the Blanchfield form, and their interaction with $0.5$-solvability of a knot.  Given a knot $K$, the rational first homology of the infinite cyclic cover of its complement $S^3 \sm \nu K$ is a $\Lambda$-module where $\Lambda := \mathbb{Q}[t,t^{-1}]$ and 
the action of $t$ is induced by the deck transformation. This module is the (rational) \emph{Alexander module} of  $K$ and is denoted by $\A(K)$.  In the language of twisted coefficients, $\A(K) = H_1(S^3 \sm \nu K; \Lambda)$.  As in \cite[\S 7B]{Rolfsen}, the Alexander module can be presented in terms of the Seifert form.  More precisely, if $F$ is a Seifert surface for $K$ and $S^3 \sm \nu F$ is its complement, then we have the following exact sequence:
\begin{equation}\label{present A0}
\begin{tikzcd}
H_1(F;\Q)\otimes_\Q \Lambda \arrow{r}{\phi}& H_1(S^3 \sm \nu F;\Q)\otimes_\Q \Lambda \arrow[r, "j"] &\A(K) \arrow[r]& 0,
\end{tikzcd}
\end{equation}
where $\phi:=(i^+_*\otimes t)-(i^-_*\otimes \Id)$ and $i^+, i^-\colon F\to S^3 \sm \nu F$ are given by pushing $F$ off of itself in the positive and negative normal directions, respectively, and $j\colon H_1(S^3 \sm \nu F;\Q)\otimes_\Q \Lambda \to \A(K)$ is induced by lifting $S^3 \sm \nu F$ to the infinite cyclic cover of $S^3 \sm \nu K$.  Pick a symplectic basis $B$ for $H_1(F;\Q)$ and let $B'$ be the basis for $H_1(S^3 \sm \nu F;\Q)$ consisting of linking duals to $B$.  With respect to these bases, $i^+_*$ is given by the Seifert matrix $V_F$ for $F$ and $i^-_*$ is given by its transpose $V_F^T$.  Thus, $\A(K)$ is the $\Lambda$-module presented by the matrix $tV_F - V_F^T$.  

There is a nonsingular, $\Lambda$-sesquilinear form $$\Bl\colon\A(K)\times\A(K)\to \Q(t)/{\Lambda}$$ called the \emph{Blanchfield form}.  It is defined via the composition of Poincar\'e duality,  the Bockstein homomorphism, and the Kronecker evaluation map \cite{Blanchfield57}.  As in \cite{Kearton75, FrPo17}, the Blanchfield form can be expressed in terms of the Seifert form.  {Consider the basis of $H_1(S^3 \sm \nu F;\Q)\otimes_\Q \Lambda$ induced by $B'$ of the previous paragraph. With respect to this basis we express elements $r,s \in H_1(S^3 \sm \nu F;\Q)\otimes_\Q \Lambda  \cong \Lambda^{2g}$ as column vectors $\vec{r}(t), \vec{s}(t)$, where $g$ is the genus of $F$.}
Then 
\begin{equation}\label{Kearton Bl formula}
\Bl(r, s) = (1-t) \cdot \vec{r}\left(t^{-1}\right)^T\cdot \left (V_F-tV_F^T\right)^{-1}\cdot \vec{s}(t).
\end{equation}

We are now ready to prove our obstruction for a link to be $0.5$-solvable concordant to a homology boundary link.

\begin{proposition}\label{prop: Blanchfield obstructs HBL}
Let $L=L_1\cup L_2$ be $0.5$-solve-equivalent to a sublink of a homology boundary link and $\widetilde{L_2}$ be a lift of $L_2$ to the infinite cyclic cover of $S^3 \sm \nu L_1$.  If $[\widetilde{L_2}]\in \A(L_1)$ is the rational homology class of this lift, then $\Bl_{L_1}([\widetilde{L_2}],[\widetilde{L_2}])=0.$
\end{proposition} 
\begin{proof}

 By Theorem~\ref{thm: HBL = BL}, whose proof appears in Section~\ref{sect: HBL = BL for 0.5}, being $0.5$-solve-equivalent to a sublink of a homology boundary link is equivalent to being $0.5$-solve-equivalent to a boundary link. Therefore, we may assume that $L$ is $0.5$-solve-equivalent to a boundary link $J=J_1\cup J_2$.   Since linking number is an invariant of $0$-solve-equivalence and vanishes for boundary links, it follows that $\lk(L_1,L_2)=0$.  Thus, $L_2$ lifts to a knot $\widetilde{L_2}$ in the infinite cyclic cover of $S^3 \sm \nu L_1$.

Next, we show that it is enough prove the proposition under the additional assumption that $J_1$ is slice.
    Let $L' = L_1' \cup L_2' := (L_1\#-J_1)\cup L_2$ and $J' =J_1' \cup J_2' := (J_1\#-J_1)\cup J_2$.  Observe that $J'$ is still a boundary link and $L'$ is $0.5$-solve-equivalent to $J'$.   Let $\widetilde{L'_2}$ be the lift of $L'_2$ in the infinite cyclic cover of $S^3 \sm \nu L'_1$. As the Blanchfield form splits under connected sum of knots,  
$$\Bl_{L_1'}([\widetilde{L_2'}],[\widetilde{L_2'}]) = \Bl_{L_1}([\widetilde{L_2}],[\widetilde{L_2}]).$$  
Hence, we may assume  that $J_1$ is slice.

Assume that $L=L_1\cup L_2$ is $0.5$-solve-equivalent to a boundary link $J=J_1\cup J_2$ and $J_1$ is slice.  In particular $L_1$ is $0.5$-solvable.
  Since $\lk(L_1,L_2)=0$, there is a Seifert surface $F$ for $L_1$ which is disjoint from $L_2$.  As in \cite[Theorem 9.2]{COT03}, if $g$ is the genus of $F$, then there is a symplectic basis $\mathcal B = \{a_1,\dots, a_g, b_1,\dots, b_g\}$ for $H_1(F)\cong \Z^{2g}$ with respect to which the Seifert form is $$\begin{pmatrix} 0&A\\B&C\end{pmatrix}$$ for some $g\times g$ square matrices $A$, $B$, and $C$.

Let $\mathcal B'=\{a_1',\dots, a_g', b_1',\dots, b_g'\}$ be the basis for $H_1(S^3 \sm \nu F)\cong \Z^{2g}$ given by linking duals to elements of $\mathcal B$. By \pref{Kearton Bl formula} the Blanchfield pairing with respect to the basis of $H_1(S^3 \sm \nu F;\Q)\otimes_\Q \Lambda$ induced by $B'$ is given by
\begin{equation}\label{Metabolizer}
\Bl(r, s) = (1-t)\vec{r}\left(t^{-1}\right)^T \begin{pmatrix}
-\left(A-tB^T\right)^{-1}\left(C-tC^T\right)\left(B-tA^T\right)^{-1}&\left(A-tB^T\right)^{-1}
\\
\left(B-tA^T\right)^{-1}&0
\end{pmatrix} \vec{s}(t).
\end{equation}
In particular, we have that $\Bl_{L_1}(j(b_k'\otimes 1), j(b_\ell'\otimes 1))=0$ for each $k,\ell$.  

Since $j\colon H_1(S^3 \sm \nu F)\to \A(K)$ is induced by lifting $S^3 \sm \nu F$ to the infinite cyclic cover of $S^3 \sm \nu K$, it is enough to show that $[L_2]\in H_1(S^3 \sm \nu F; \Q)$ is in the span of $b_1',\dots, b_g'$.  Moreover, since the basis $\mathcal B'$ is the linking dual of $\mathcal B$, it is enough to show that $\lk(L_2, a_i)=0$ for each~$i$.  To do so, we will find suitable rational 2-chains in a 0.5-solution for $L_1$ bounded by $L_2$ and $a_i$.

Let $(W,C=C_1\cup C_2)$ be a $0.5$-solve-equivalence from $L$ to $J$ and $\Delta\subseteq B^4$ be a slice disk for $J_1$.    Capping off the boundary component of $W$ containing $J$ with $B^4$ and capping off the boundary component of $C_1$ containing $J_1$ with $\Delta$, we produce a 4-manifold $W\cup B^4$ containing a disk $D=C_1\cup \Delta$ bounded by $L_1$.  As in the proof of Proposition~\ref{prop: solvability in terms of disks}, we see that $W':=(W\cup B^4) \sm \nu D$ is a 0.5-solution for $L_1$.

If $X_1,\dots, X_k, Y_1, \dots, Y_k\subseteq W$ are the 1-Lagrangians and $0$-duals for the $0.5$-solve-equivalence $(W,C)$, then they also serve as $1$-Lagrangians and 0-duals in the 0.5-solution $W'$. Let $\widehat{F}\subseteq \bdry W' = M_{L_1}$ be the closed surface given by capping the Seifert surface $F$ for $L_1$ with a pushoff of the disk $D$.  Since $\widehat F$ is disjoint from the Lagrangians $X_i$ and the duals $Y_i$, it follows that $[\widehat F]=0$ in $H_2(W')$.  Let $R\subseteq W'$ be a compact oriented 3-manifold bounded by $\widehat F$.  

Next we follow the same strategy as used in  \cite[Theorem 9.2]{COT03}, in order to separate $R$ from each of $X_1, \dots, X_k$.  First, arrange that $R$ is transverse to $X_i$, so that $R\cap X_i=\alpha$ is a 1-manifold.  Since $R$ is dual to the generator for $H_1(W')$ and since each $X_i$ lifts to the infinite cyclic cover, $[\alpha]=0$ in $H_1(X_i)$.  Thus, $\alpha$ is the oriented boundary of a collection of nested surfaces $S$ in $X_i$. Let $S'$ be the closure of an innermost component of  $S'$.  Remove from $R$ a neighborhood of the boundary of $S'$ and replace it with the boundary of a tubular neighborhood of $S'$.  This procedure, referred to as ambient surgery, reduces the number of components of $X_i\cap R$.  Iterate until $X_i\cap R = \emptyset$ for all $i$. 

A standard ``half lives, half dies'' argument reveals there are nonseparating disjoint simple closed curves $a_1,\dots, a_g\subseteq F$ which bound rational 2-chains  $A_1,\dots, A_g$ in $R$.  See for example \cite[\S 8E, Lemma 18]{Rolfsen}. Let $A_i^+$ be the result of pushing $A_i$ off $R$ in the positive normal direction.  These are used in \cite[Theorem 9.2]{COT03} to show that $\lk(a_i, a_j^+) =0$ by arguing that the linking number agrees with the signed count of intersections points, denoted by $A_i\cdot A_j^+$, which is zero since $A_i$ and $A_j^+$ are disjoint. We use a similar strategy to show that $\lk(L_2, a_i) = 0$ for each $i$.

As $J$ is a boundary link, let $G_1$ and $G_2$ be a pair of disjoint surfaces bounded by $J_1$ and $J_2$.  Let $\Sigma = C_2\cup G_2\subseteq W'$.  As $G_1$ and $G_2$ are disjoint, $\pi_1(G_2)$ sits in $\pi_1(S^3 \sm \nu J_1)^{(1)}$.  It follows that $\pi_1(\Sigma)$ sits in $\pi_1(W')^{(1)}$.  Additionally, since $C_2$ is disjoint from the Lagrangians and duals $X_1,Y_1, \dots, X_k, Y_k$, which form a basis for $H_2(W')$, it follows that $[\Sigma]=0$ in $H_2(W', \bdry W')$.  

In the following paragraph, we arrange that $R$ is disjoint from $\Sigma$.  We do so by  imitating the ambient surgery argument used in \cite[Theorem 9.2]{COT03} to get $R \cap X_i=\emptyset$ for each $i$. We isotope if needed to arrange that $\Sigma$ intersects $R$ transversely, so  $\Sigma\cap R=\alpha$ is a 1-manifold.  If $[\alpha]\neq 0$ in $H_1(\Sigma)$, then there is some $[\beta]\in H_1(\Sigma)$ with $\alpha\cdot \beta\neq 0$.  When regarded as an element of $H_1(W)$, $[\beta]\cdot [R]\neq 0$.   This contradicts the fact that $\pi_1(\Sigma)\subseteq \pi_1(W')^{(1)}$. Thus $\alpha$ is the oriented boundary of a collection of nested surfaces $S$ in $\Sigma$.  These surfaces, being surfaces with boundary, have trivial normal bundles in $W'$. Just as we did with $X_i$, we use these to modify $R$ by ambient surgery.  More precicely, if $S'$ is the closure an innermost component of $S$ then remove from $R$ a neighborhood of $\bdry S'$ and replace it with the boundary of a tubular neighborhood of $S'$.  Iterate until $R\cap \Sigma=\emptyset$.

For each $i$, since $R\cap \Sigma=\emptyset$ and $A_i\subseteq R$, we have $\Sigma\cap A_i=\emptyset$.  Let $A_i'\subseteq S^3 \sm L_1$ be a Seifert surface for $a_i$.  Then $[A_i\cup -A_i']\in H_2(W';\Q)$.  Since $[\Sigma]=0$ in $H_2(W', \bdry W')$, we have
$$
0=\Sigma \cdot [A_i\cup -A_i'] = \Sigma\cdot A_i - L_2\cdot A_i' = 0-\lk(L_2, a_i).
$$
As $\lk(L_2, a_i)=0$ for each $i$, we conclude that $\Bl_{L_1}([\widetilde{L_2}], [\widetilde{L_2}])=0$. \end{proof}

Now, we prove Theorem~\ref{thm: n=0}, whose statement we recall.

\begin{reptheorem}{thm: n=0}
For each positive integer $k$, there exists a 0-solvable link $L$ with slice components such that
\begin{enumerate}[font=\upshape]
\item\label{item:vanishingMuFor0.5} $\mu_I(L) = 0$ for each multi-index $I$ with $|I|< k$, and 
\item $L$ is not $0.5$-solve-equivalent to any sublink of a homology boundary link.
\end{enumerate}
\end{reptheorem}
\begin{proof}

We will recursively construct a sequence of 2-component links $L^k = L^k_1\cup L^k_2$ with the following properties. First, the link $L^k$ has slice components and vanishing Milnor's invariants of length less than $k$ for each $k$. Moreover, $L^k_2$ lifts to a knot $\widetilde {L^k_2}$
 in the infinite cyclic cover of $S^3 \sm L^k_1$ and its rational homology class $[\widetilde {L^k_2}]\in \A(L^k_1)$ satisfies $\Bl_{L^k_1}([\widetilde {L^k_2}],[\widetilde {L^k_2}])\neq 0.$

 For the base case, we take $L^4$ to be the link in Figure~\ref{fig: 0.5 example}.  Note that each component of $L^4$ is slice and $L^4$ has vanishing Milnor's invariants of length less than $4$ since it has vanishing pairwise linking number. Moreover, note that $[\widetilde{L_2^4}]$ generates $\A(L_1^4)$. The nonsingularity of the Blanchfield form implies that $\Bl_{L_1^4}([\widetilde{L_2^4}],[\widetilde{L_2^4}])\neq 0$. 
 
 Suppose now that we have constructed $L^4, \dots, L^k$.  According to \cite[Theorem 3.3]{Cochran91} there exists a Brunnian link $B^k=B^k_1\cup B^k_2$ with vanishing Milnor's invariants of length less than $k$ and for which all Milnor's invariants of length $k$ agree with those of $L^k$. Define $L^{k+1}$ to be an exterior band sum of $L^k$ and $-B^k$, the reverse of the mirror image of $B^k$. Since first non-vanishing Milnor's invariants are additive under exterior band sum \cite[Theorem 8.12]{C4}, $L^{k+1}$ has vanishing Milnor's invariants of length less than $k+1$.

As $L_1^{k+1}$ is the result of banding together $L_1^{k}$ and an unknot, it follows that $L_1^{k+1}$ is isotopic to $L_1^{k}$, and so is slice.  We will now conflate $L_1^{k}$ and $L_1^{k+1}$. Similarly, $L_2^{k+1}$ is also slice.  Since $\lk(B^k_1,B^k_2)=0$ and $B^k_1$ is unknotted, $B^k_2$ is nullhomotopic in $S^3 \sm B^k_1$.  Thus, $L_2^{j+1}$ is homotopic to $L_2^{k}$ in $S^3 \sm L_1^{k} = S^3 \sm L_1^{k+1}$.  In particular, for a suitable choice of lifts, we have that $[\widetilde{L_2^k}] = [\widetilde{L_2^{k+1}}]$ in $\mathcal{A}(L_1^{k}) = \mathcal{A}(L_1^{k+1}).$  We conclude that 
$$\Bl_{L_1^k}([\widetilde{L_2^k}],[\widetilde{L_2^k}]) =\Bl_{L_1^{k+1}}([\widetilde{L_2^{k+1}}],[\widetilde{L_2^{k+1}}]) \neq 0.$$  
Thus, by Proposition~\ref{prop: Blanchfield obstructs HBL}, $L$ is not $0.5$-solve-equivalent to a sublink of a homology boundary link.  

  We obtained a link $L^{k}$ with slice components which has vanishing Milnor's invariants for each multi-idex $I$ with $|I|<k$ and is not $0.5$-solve-equivalent to any sublink of a homology boundary link. Once $k>4$, the $0$-solvability of $L^{k}$ follows from \cite[Theorem 1]{Martin22}.
\end{proof}

 \section{C-complexes, metabolic linking forms, and 0.5-solvability} \label{sect: C-cplx}
 
  Our proof of Theorem~\ref{thm: HBL = BL} requires the theory of C-complexes and their linking forms, first defined by Cooper \cite{Cooper82} and studied extensively by Cimasoni \cite{Cimasoni04}.  Many invariants which obstruct $0.5$-solvability can be computed in terms of a C-complex, see for example \cite{Cooper82, Cimasoni04, CiF, Conway18, MelMel03, DNOP20}, amongst others.  We are interested in a result of the opposite form.  We  produce a condition on a C-complex which {implies} 0.5-solvability. First, we recall the definition of a C-complex.
  
  \begin{figure}[h]
\begin{picture}(150,80)
\put(0,0){\includegraphics[height=.1\textheight]{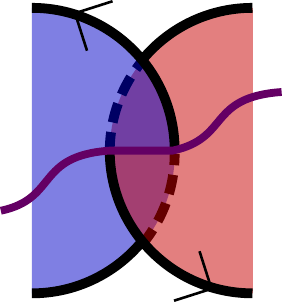}}
\put(90,0){\includegraphics[height=.1\textheight]{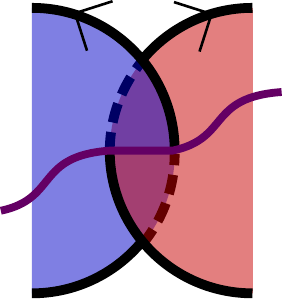}}

\end{picture}\vspace{.5cm}
\caption{A positive (left) and negative (right) clasp intersection with a loop passing over it.}\label{fig:clasp}
\end{figure} 

\begin{definition}[\cite{Cimasoni04}]
Let $L=L_1\cup\dots\cup L_n$ be an $n$-component link and $F=F_1\cup\dots\cup F_n$ be a union of compact oriented smoothly embedded surfaces in $S^3$. We say $F$ is a \emph{C-complex} for $L$ if it satisfies the following:
\begin{itemize}
\item for any $i\neq j$, surfaces $F_i$ and $F_j$ intersect transversely in a (possibly empty) union of embedded arcs each having one endpoint in $\bdry F_i$ and the other in $\bdry F_j$.  Such an arc is called a \emph{clasp}.
\item $F_i\cap F_j\cap F_k = \emptyset$ for all $i<j<k$. 
\item $\bdry F_i = L_i$ for each $i$.
\end{itemize}\end{definition}

A clasp $c$ in a C-complex $F$ has a \emph{sign}, denoted by $\sigma(c)\in \{\pm 1\}$, which is determined by the sign of either of the intersections between $L$ and $F$ at the end points of $c$  (see Figure~\ref{fig:clasp}).  A closed curve $\ell$ (possibly not embedded) sitting on a C-complex is called a \emph{loop} if the intersection of a neighborhood of each clasp with $\ell$ is a (possibly empty) union of embedded arcs each intersecting the clasp as in Figure~\ref{fig:clasp}. Given a loop $\ell$ on a C-complex $F=F_1\cup\dots\cup F_n$ and a tuple $\epsilon = (\epsilon_1,\dots, \epsilon_n)\in \{\pm1\}^n$, a new closed curve $\ell^\epsilon\subseteq S^3 \smallsetminus F$ is obtained by pushing the part of $\ell$ sitting in $F_i$ off in the $\epsilon_i$-normal direction and interpolating at each clasp.  See for example~\cite{CiF} for a more detailed discussion.  Any closed curve curve on a C-complex is homotopic to a loop. Hence for any $\epsilon = (\epsilon_1,\dots, \epsilon_n)\in \{\pm1\}^n$, there is a linking form $V^\epsilon\colon H_1(F)\times H_1(F)\to \Z$ given by setting $V^\epsilon([\alpha], [\beta]) = \lk(\alpha, \beta^\epsilon)$, when $\alpha$ and $\beta$ are embedded loops in $F$ and then extending it linearly.

We are now ready to present our sufficient condition for $0.5$-solvability.  

\begin{definition}\label{defn: pairing}
Let $F = F_1\cup \dots\cup F_n$ be a C-complex and $c^+$ and $c^-$ be clasps of opposite sign in $F_i\cap F_j$.   An embedded loop $\ell$ in a C-complex $F$ is said to \emph{pair} $c^+$ and $c^-$ if 
\begin{itemize}
\item $\ell$ passes once over $c^+$ and once over $c^-$, 
\item $\ell$ is disjoint from all other clasps, and 
\item $\ell\cap F_i$ and $\ell\cap F_j$ separate $F_i$ and $F_j$, respectively.
\end{itemize} A collection of disjoint loops $\ell_1\cup \dots \cup \ell_m$ is called a \emph{complete pairing} for $F$ if each $\ell_i$ pairs a positive and negative clasp in $F$ and every clasp is contained in exactly one $\ell_i$.  
\end{definition}

\begin{definition}\label{defn: metabolic}
A C-complex $F=F_1\cup\dots\cup F_n$ is said to admit a \emph{metabolic linking form}  if
\begin{itemize}
\item 
There is a nonseparating collection of disjoint embedded simple closed curves $\{\alpha_i\}_{1\leq i \leq g}$ where $g =g(F) = \sum_{i=1}^n g(F_i)$ is the sum of the genera of the component of $F$.  
\item There is a complete pairing $\ell_1\cup \dots \cup \ell_m$ for $F$ disjoint from $\alpha_1\cup \dots \cup \alpha_g$.
\item For each $\epsilon\in \{\pm1\}^n$, the linking form $V^\epsilon$ vanishes on  $$\left\{[\alpha_i], [\ell_j]\right\}_{1\leq i \leq g, 1\leq j \leq m} \subseteq H_1(F).$$
\end{itemize}
The union of curves $\alpha_1\cup \dots \cup \alpha_g$ forms a link which we will call a \emph{derivative} for $L$. 
\end{definition}
The use of the word ``derivative'' in Definition~\ref{defn: metabolic} is inspired by \cite{derivatives}.  For illustration purposes, an example appears in Figure~\ref{fig:derivative}.

   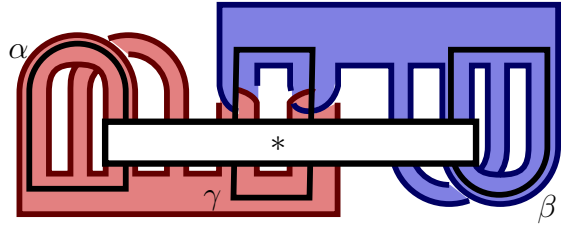
\begin{figure}[h]
\begin{picture}(210,110)

\put(0,0){\includegraphics[width=.45\textwidth]{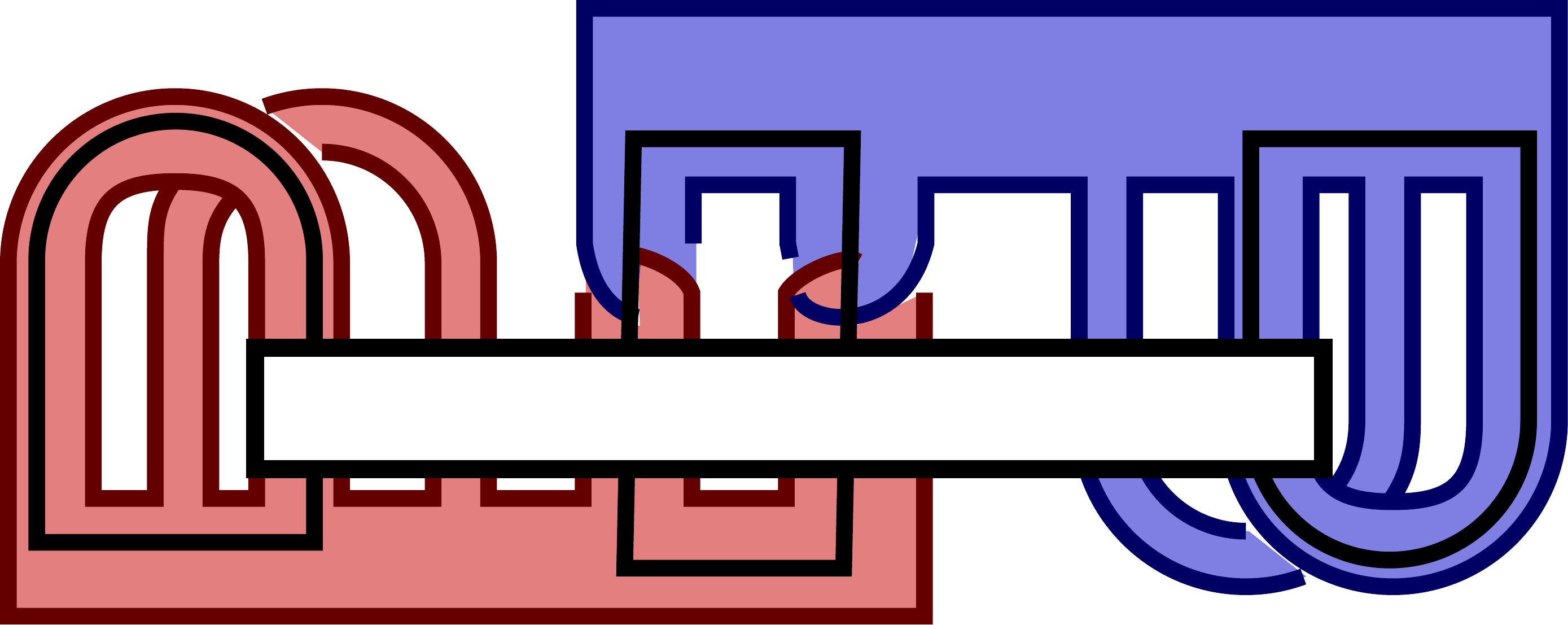}}
\put(-3,60){\large$\alpha$}
\put(195,0){\large $\beta$}
\put(70,5){\large$\gamma$}
\put(95,25){\large $*$}
\end{picture}\vspace{.5cm}
\caption{A C-complex $F$ with curves $\alpha, \beta, \gamma$.  {The $*$ indicates an arbitrary framed 6-component string link.}  The curve $\gamma$ forms a pairing for $F$.  If the linking numbers between $\alpha$, $\beta$, $\gamma$ and their pushoffs $\alpha^\epsilon$, $\beta^\epsilon$, $\gamma^\epsilon$ all vanish then $F$ has metabolic linking form.  }\label{fig:derivative}
\end{figure}  

\begin{theorem}\label{thm: metabolic C-cplx implies 0.5-sol}
A link which admits a C-complex with metabolic linking form is $0.5$-solvable.
\end{theorem}
\begin{proof}
Let $L$ be a link with  a C-complex $F$ which has a metabolic linking form.  Let $D$ be a derivative and $P$ a complete pairing as in Definition~\ref{defn: metabolic}. We will use the fact that $S:=D\cup P$ is a link with vanishing pairwise linking numbers to guide a sequence of double-delta moves which reduces $L$ to a slice link.  This will complete the proof since the 0.5-solvability of a link is preserved under the double-delta move~\cite[Proposition 7.2]{MartinThesis}.  See Figure~\ref{fig:doubleDelta} for the delta and double-delta moves.

   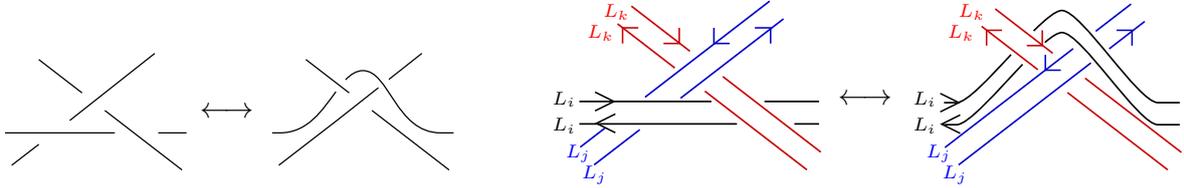
\begin{figure}[h]
\begin{picture}(580,75)
\put(0,0){\includegraphics[width=.15\textwidth]{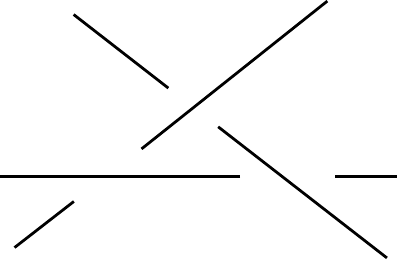}}
\put(73,20){$\longleftrightarrow$}
\put(100,0){\includegraphics[width=.15\textwidth]{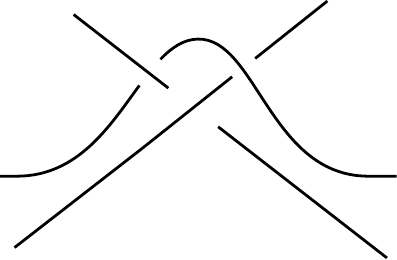}}

\put(215,0){\includegraphics[width=.2\textwidth]{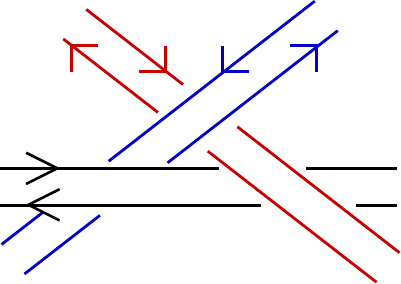}}
\put(312,25){$\longleftrightarrow$}
\put(350,0){\includegraphics[width=.2\textwidth]{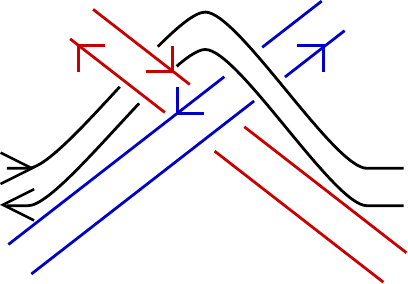}}

\put(205,15){\tiny{$L_i$}}
\put(205,25){\tiny{$L_i$}}
\put(210,5){\textcolor{blue}{\tiny{$L_j$}}}
\put(216,-3){\textcolor{blue}{\tiny{$L_j$}}}
\put(218,50){\textcolor{red}{\tiny{$L_k$}}}
\put(224,58){\textcolor{red}{\tiny{$L_k$}}}

\put(340,15){\tiny{$L_i$}}
\put(340,25){\tiny{$L_i$}}
\put(345,5){\textcolor{blue}{\tiny{$L_j$}}}
\put(351,-3){\textcolor{blue}{\tiny{$L_j$}}}
\put(353,50){\textcolor{red}{\tiny{$L_k$}}}
\put(357,58){\textcolor{red}{\tiny{$L_k$}}}
\end{picture}
\vspace{.3cm}
\caption{Left:  The delta move.  Right: The double delta move, it is not required that $i$, $j$, and $k$ be distinct.  }\label{fig:doubleDelta}
\end{figure}

Our argument passes through the setting of string links.  For each component $\alpha_i$ of the derivative $D$, let $\Delta_i$ be a disk disjoint from $F$ except that it intersects $F$ in an arc which crosses $\alpha_i$ transversely in a single point and is otherwise disjoint from $S$.  Similarly, for each component $\ell_k$ of the complete pairing $P$, let $E_k$ be a disk intersecting $F$ in an arc crossing $\ell_k$ in a single point with positive orientation and which is otherwise disjoint from $S$.  Connect these disks $\{\Delta_i, E_k\}_{i,k}$ together with bands disjoint from $F$ to arrive at a single disk denoted by $\Delta$.
Note that $\Delta$ has a neighborhood $N(\Delta)\cong D^2\times [0,1]$ and we see that $S\cap N(\Delta) \subseteq N(\Delta)$ is a trivial string link and $F \cap N(\Delta)$ is the union of bands about the components of this string link.  The complement $ S^3\smallsetminus \interior\left(N(\Delta)\right)$ is likewise a copy of $D^2\times[0,1]$ so that $S \cap  \left(S^3\smallsetminus \interior\left(N(\Delta)\right)\right) \subseteq S^3\smallsetminus \interior\left(N(\Delta)\right)$ is also a string link, which we denote by $T$.  Since $S$ has vanishing pairwise linking numbers so does $T$ as does $T^{-1}$, the horizontal mirror image  of $T$ with orientation reversed.  By \cite[Theorem C]{NaSt03} there is a sequence of delta-moves transforming the trivial string link $S \cap N(\Delta)\subseteq N(\Delta)$ to $T^{-1}$.  By performing these delta moves on the bands $F\cap N(\Delta)$ about the components of  the trivial string link, we see a sequence of double-delta moves changing the link $L$ to a new link $L'$ which admits a new C-complex $F'$ with metabolic linking form. Moreover, there is a new pairing $P'$ and derivative $D'$ which are obtained by performing delta moves on $P$ and $D$, respectively.  The important consequence is that $S':=D' \cup P'$ is the closure of the slice string link $T*T^{-1}$ so that  $S'$ is slice.

Next, we explain why $S'=D'\cup P'$ being slice implies that $L'$ is slice.  For each component $\ell'_k$ of $P'$, consider a neighborhood of $\ell'_k$ in $F'$ consisting of a pair of disks intersecting in two clasps bounded by  $BD(\ell'_k)$, the Bing double of $\ell'_k$.  Push these disks off of the boundary of $F'$ in order to make them interior to $F'$.  By pushing them only a small distance, we arrange that the clasps paired by $\ell_k'$ are contained in the union of their interiors.  The union of the resulting disks is denoted by $N(\ell'_k)$. The key observation is that since each clasp in $F$ is contained in $N(\ell'_k)$ for some $k$, $F'\smallsetminus \interior \left(\Cup_k N\left({\ell'_k}\right)\right)$ has no clasps and so is a disjoint union of surfaces each bounded by a single component of $L'$ as well as several components of $BD(P'):=\Cup_k BD(\ell'_k)$. In Figure~\ref{fig:pairingCut} we see a pairing loop $\ell$, its neighborhood $N(\ell)$ and $F\setminus N(\ell)$. 

\vspace{.4cm}
  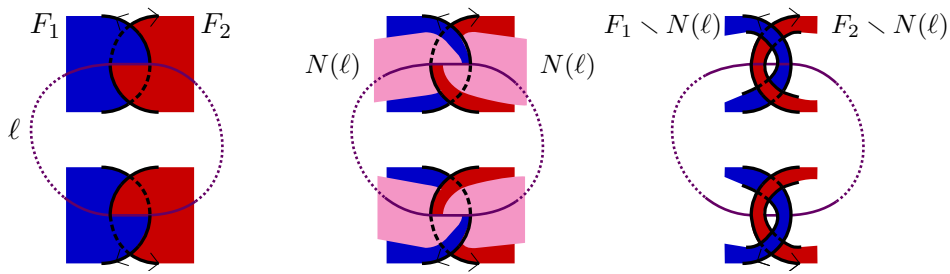
\begin{figure}[h]
\begin{picture}(210,105)
\put(-40,0){\includegraphics[height=.15\textheight]{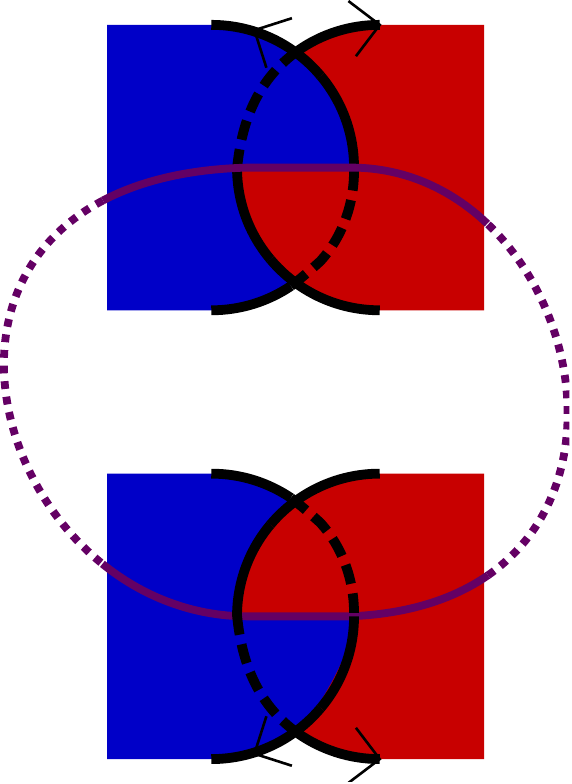}}
\put(-48,50){$\ell$}
\put(-40,90){$F_1$}
\put(24,90){$F_2$}

\put(80,0){\includegraphics[height=.15\textheight]{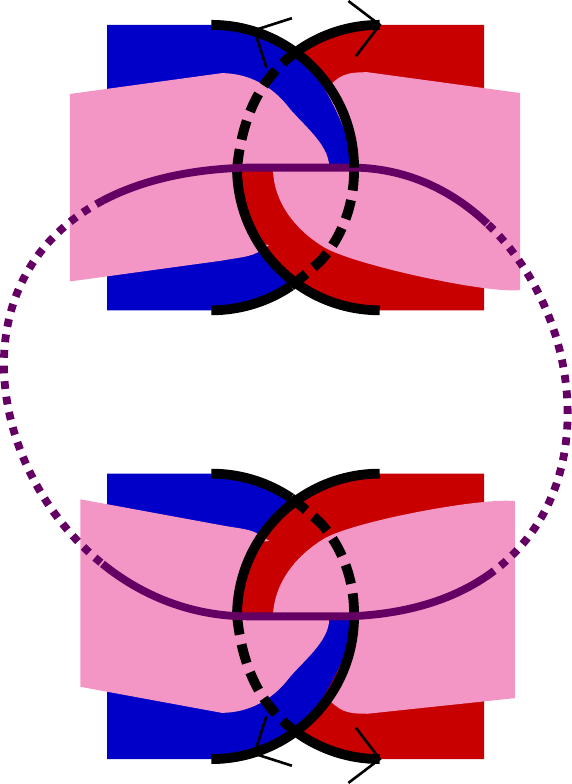}}
\put(63,75){\small{$N(\ell)$}}
\put(151,75){\small{$N(\ell)$}}

\put(200,0){\includegraphics[height=.15\textheight]{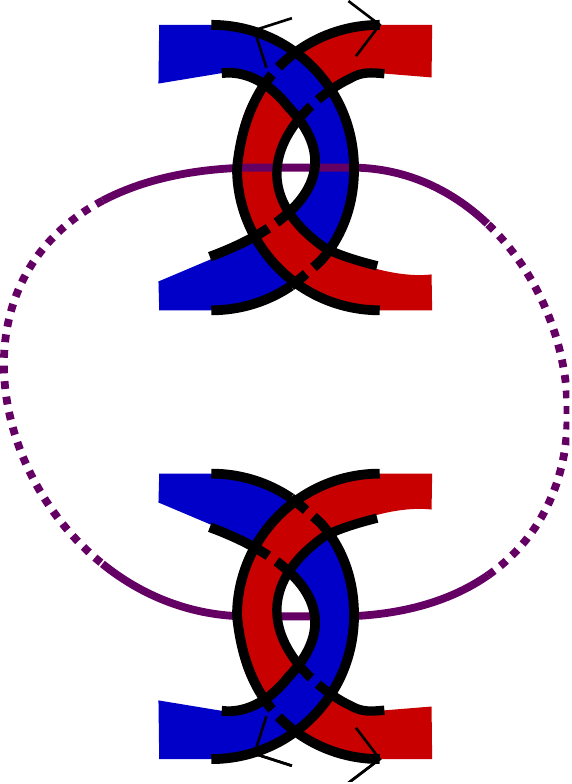}}
\put(175,90){\small{$F_1\smallsetminus N(\ell)$}}
\put(260,90){\small{$F_2\smallsetminus N(\ell)$}}
\end{picture}\vspace{.5cm}
\caption{Left to right:  A loop $\ell$ pairing two clasps in a C-complex $F$.  A neighborhood $N(\ell)$ consisting of two disks intersecting in two clasps.   $F\smallsetminus N(\ell)$ has two new boundary components and two fewer clasps.  }\label{fig:pairingCut}
\end{figure}

  Similarly, for any $\alpha'_i \in D'$, let $N(\alpha'_i)$ be an annular neighborhood of $\alpha'_i$ in $F'$.  That is, $N(\alpha'_i)\subseteq F'$ is an annulus bounded by a copy of $\alpha'_i$ along with a 0-framed reversed pushoff.  This link is denoted by $\double(\alpha'_i)$ and $Z(D') := \Cup_i \double(\alpha'_i)$.  We now have that  $F'\smallsetminus\interior\left(\Cup_{j} N({\ell'_j}) \cup \Cup_{i} N({\alpha'_i})\right)$ is a disjoint union of embedded planar surfaces each of which is bounded by one component of $L'$ together with some components of $BD(P')\cup \double(D')$.  
  In the situation of Figure~\ref{fig:derivative} the resulting surface appears in Figure~\ref{fig:CcxCut}.   
  We have already arranged that $P'\cup D'$ is slice, so that $\Cup_{j} BD(\ell'_j) \cup \Cup_{i} \double(\alpha'_i)$ is slice.  This planar surface together with these slice disks results in  results in a collection of slice disks for $L'$.  
\vspace{.8cm}
  
  \begin{figure}[h]
\begin{picture}(210,80)
\put(0,0){\includegraphics[width=.45\textwidth]{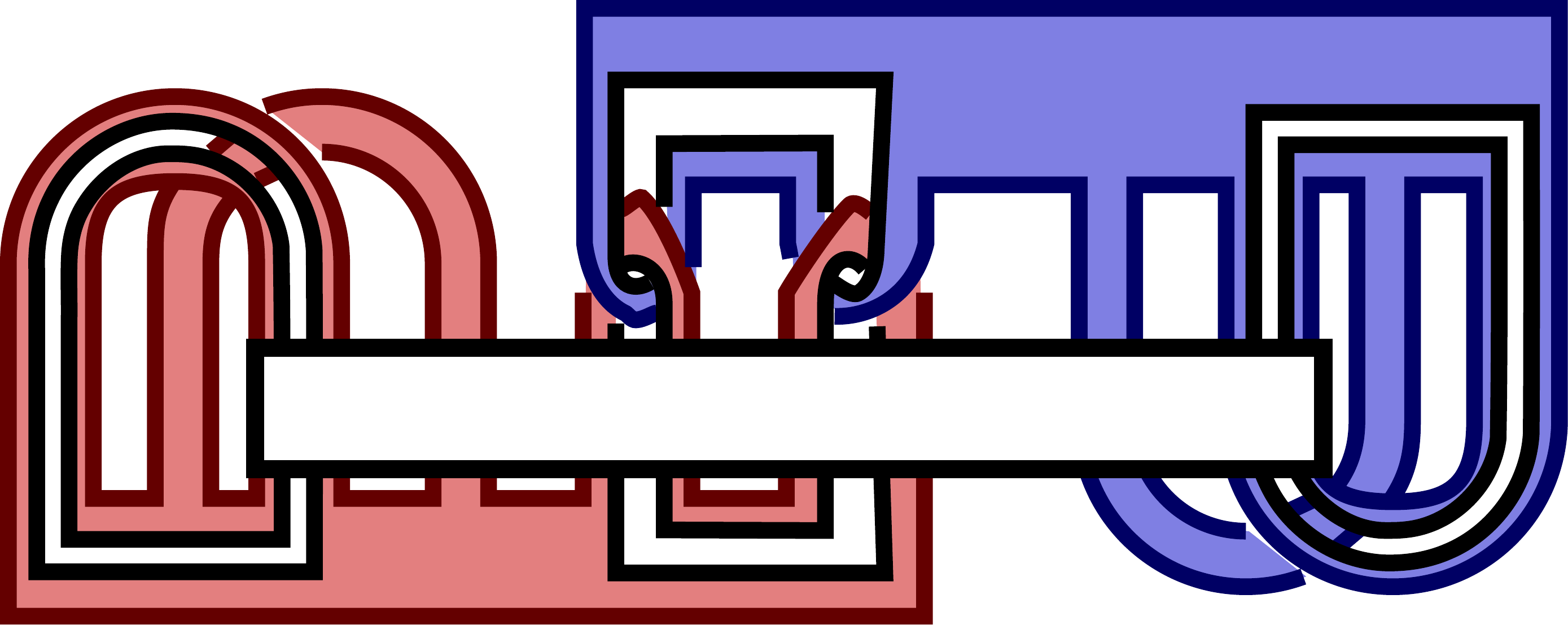}}
\put(95,25){\large $*$}
\end{picture}\vspace{.5cm}
\caption{For the C-complex, derivative and pairing or Figure~\ref{fig:derivative}, here is the result of cutting out a neighborhood of the derivative and paring curves.}\label{fig:CcxCut}
\end{figure}

 Finally, since $L'$ is slice, it is $0.5$-solvable.  Since $L$ and $L'$ differ by a sequence of double-delta moves, $L$ is $0.5$-solvable by \cite[Proposition 7.2]{MartinThesis}, completing the proof.\end{proof}

\section{Homology boundary links, fusions of boundary links, and $0.5$-solve-equivalence.}\label{sect: HBL = BL for 0.5}

In this section, we prove Theorem~\ref{thm: HBL = BL}.  We will show that every sublink of a homology boundary link is $0.5$-solve-equivalent to a boundary link.  Since the argument employs many {varied} techniques, we start with an outline.  As we explore in Subsection~\ref{subs: Fusion Seifert form},  every homology boundary link is concordant to a fusion of a boundary link \cite[Corollary 4.1]{CocLev91}.  A fusion of a  boundary link bounds an immersed union of surfaces with a particular intersection pattern.  We call these fusion complexes.  The definition of the Seifert form for a boundary link extends to give a well defined linking form on a {fusion Seifert surface}.  The failure of a fusion Seifert surface to be embedded is invisible to the Seifert form, meaning for any fusion Seifert form there is a boundary link with an identical form. In Subsection~\ref{metabolic to 0.5 solvable}, we introduce a notion of a fusion complex admitting a metabolic linking form and use Theorem~\ref{thm: metabolic C-cplx implies 0.5-sol} to prove that it implies 0.5-solvability.  In Subsection~\ref{metabolic fusion}, we complete the proof of Theorem~\ref{thm: HBL = BL} by taking the band sum of a fusion of boundary links with the reverse of the mirror image of a boundary link with identical linking form and showing that the resulting linking form is metabolic.

\subsection{Fusions of boundary links and their Seifert form}\label{subs: Fusion Seifert form}

By \cite[Corollary 4.1]{CocLev91} every homology boundary link is concordant to a fusion of a boundary link.  We begin by explaining what fusion means. Let $L=L_1\cup\dots\cup L_m$ be a link, $I=[0,1]$ be an interval, and $\delta\colon I\times I\into S^3$ be an embedding with 
\begin{itemize}
\item $\delta(I\times I)\cap L_i = \delta(I\times \{0\})$ and $\delta(I\times I)\cap L_j = \delta(I\times \{1\})$  for some $i,j$ with $i\neq j$,
\item $\delta(I\times I)$ otherwise disjoint from $L$,
\item the orientation of $\delta(I\times\{0\})$ agrees with the orientation of $L_i$ and $\delta(I\times\{0\})$ disagrees with the orientation of $L_j$.
\end{itemize}
Let $L_\delta$ be the result of changing $L$ by replacing $\delta(I\times\{0, 1\})$ by $\delta(\{0,1\}\times I)$ (see Figure~\ref{fig:Fusion}).  We say that $\delta$ is a  \emph{fusion band} and $L_\delta$ is the result of \emph{fusion} along $\delta$. The result of this construction is called the \emph{fusion of $L$ along $\delta$}.

  \begin{figure}[h]
\begin{tikzpicture}
\node at (-3,0){\includegraphics[height=.1\textheight]{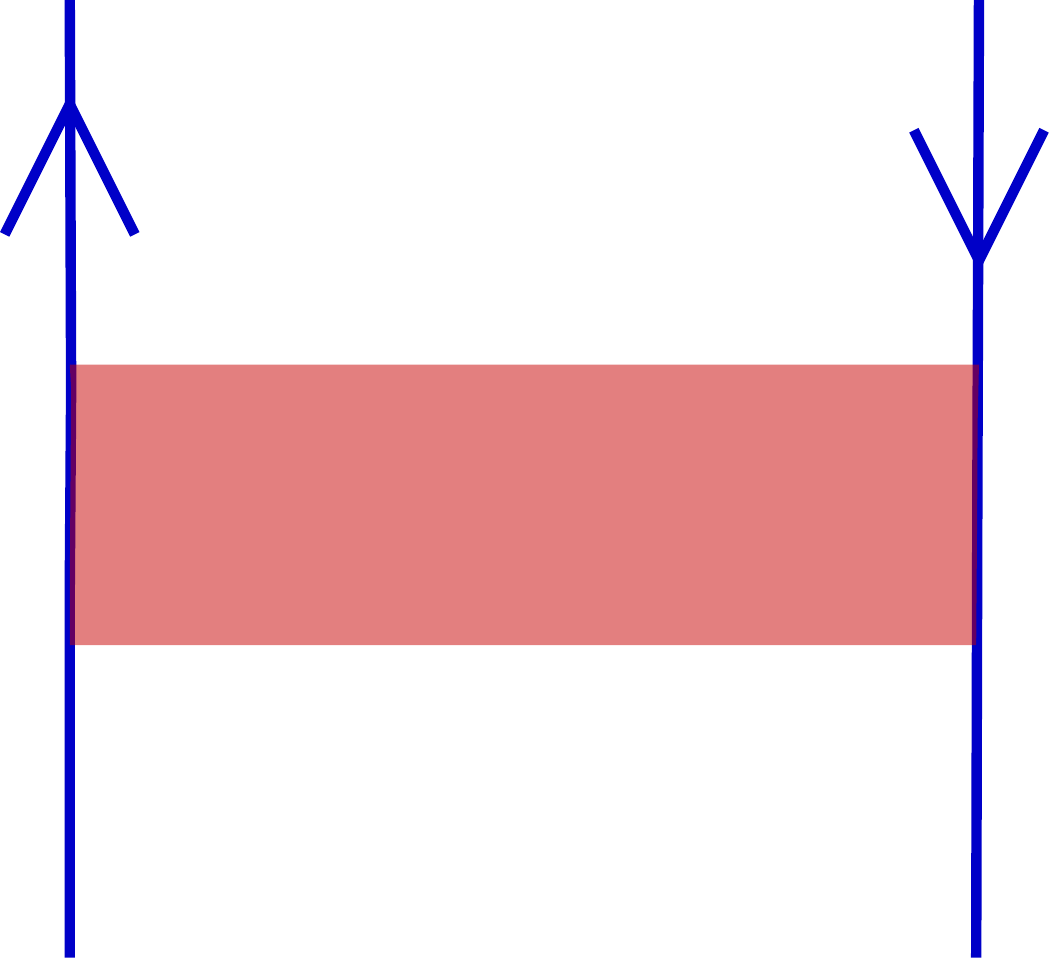}};
\node at (-3.85,1){$L_i$};
\node at (.85-3,1){$L_j$};
\node at (-3,0){$\delta$};
\node at (.8,0){$\phantom{1}$};
\end{tikzpicture} 
\begin{tikzpicture}
\node at (3,0){\includegraphics[height=.1\textheight]{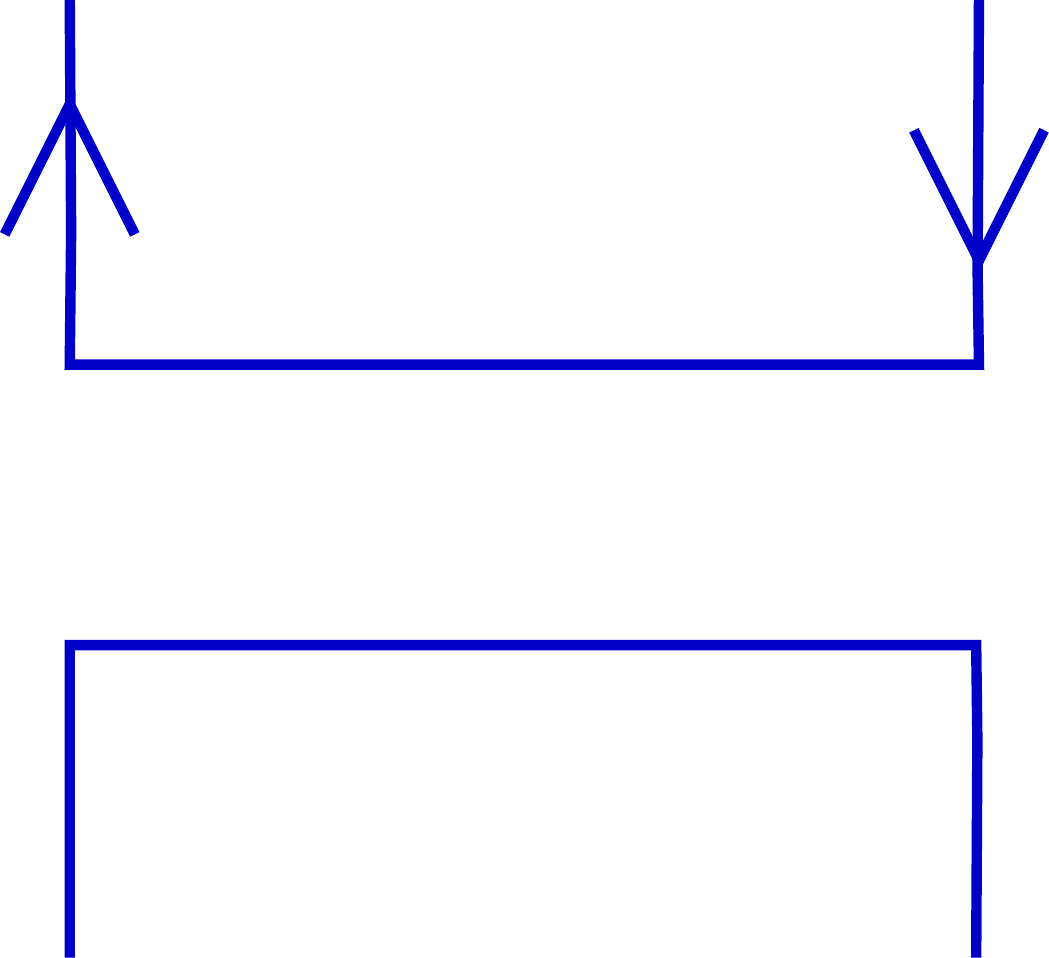}};
\end{tikzpicture}

\caption{Left: A link $L$ and a fusion band $\delta$  ($i\neq j$)  Right: the result of fusion, $L_\delta$}\label{fig:Fusion}
\end{figure}
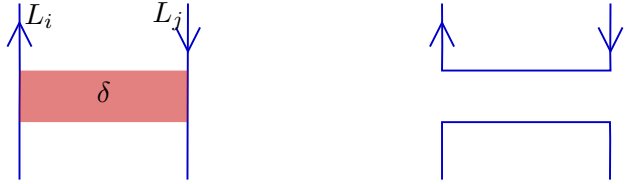 
 
Let $L=L^0, L^1,\dots, L^n$ be a sequence of links with $L^i$ obtained from $L^{i-1}$ from a fusion along a band $\delta_i$.  Then $\delta = \{\delta_1,\dots,\delta_n\}$ is a \emph{collection of fusion bands} for $L$ and $L^n = L_\delta$ is the result of fusion of $L$ along $\delta$.  If $L$ is a boundary link, then $L^n$ is a \emph{fusion of a boundary link}.  Note that up to isotopy we may assume that these bands are all disjoint from each other.

  \begin{figure}[h]
\begin{picture}(240,110)
\put(-20,0){\includegraphics[height=.15\textheight]{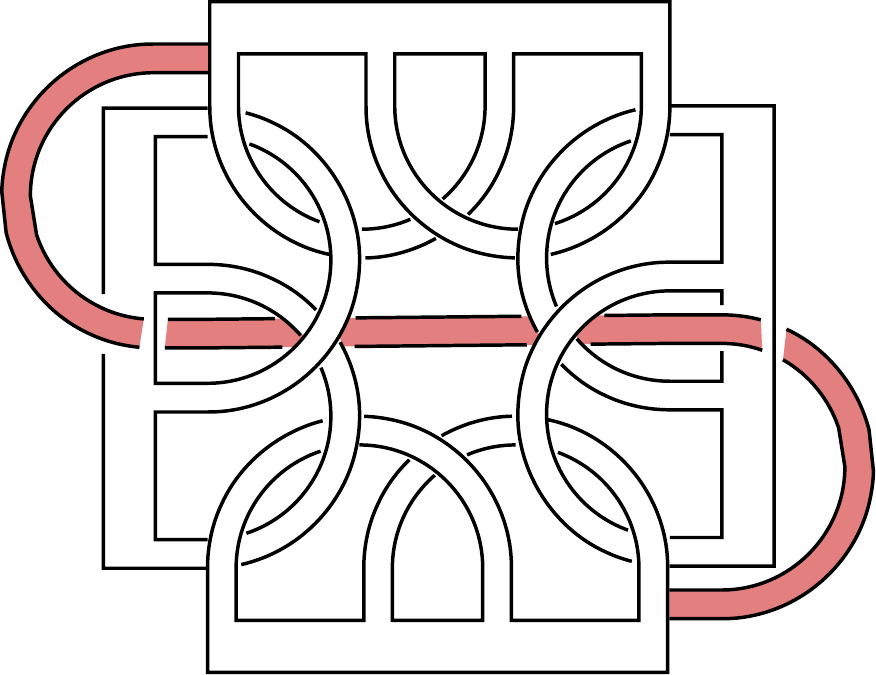}}
\put(-25,70){$\delta$}
\put(140,0){\includegraphics[height=.15\textheight]{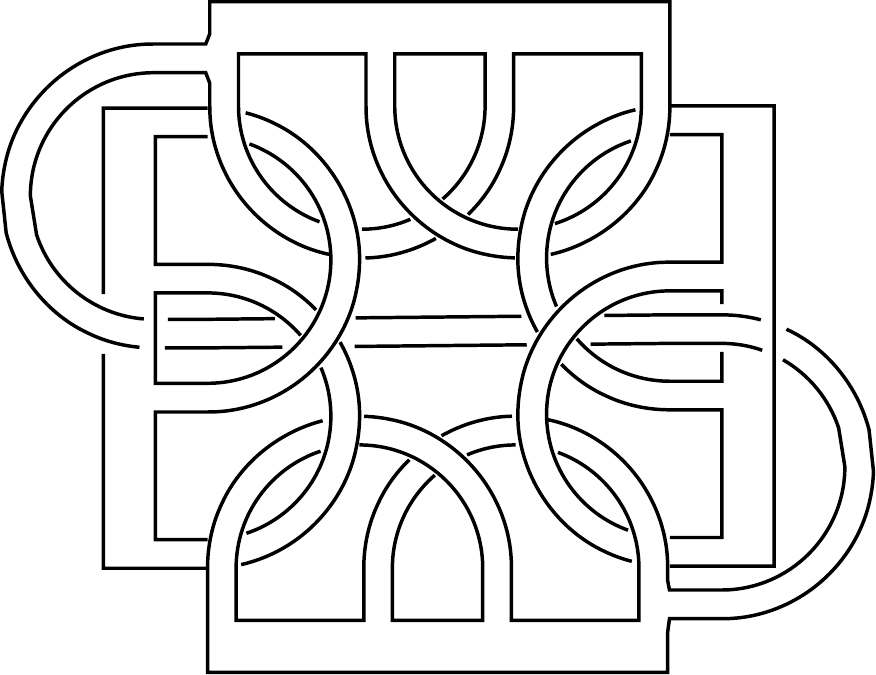}}

\end{picture}\vspace{.4cm}
\caption{Left: A boundary link $L$ and a fusion band $\delta$.  Right: The resulting fusion of a boundary link $L_\delta$}\label{fig:FBL}
\end{figure}

Let $L$ be a boundary link, $G$ be a collection of disjoint Seifert surfaces for $L$, and $\delta$ be a collection of fusion bands. After an isotopy we may assume $\delta$ intersects the interior of $G$ transversely in a collection of arcs each of the form $\delta_k(I\times \{t\})$ for some $t\in (0,1)$.  The union of $G$ with the image of $\delta$ produces an immersed surface with boundary $L_\delta$ with particular  self-intersections, as in the following definition.

  \begin{definition}\label{defn: fusion surface}
  Let $\Sigma$ be a compact oriented surface.  Suppose that each component of $\Sigma$ has a single boundary component
  and that $\Sigma$ contains disjoint oriented embedded arcs $a=a_1\cup\dots\cup a_{\ell}$ and $b= b_1\cup\dots\cup b_{\ell}$ interior to $\Sigma$ except that the boundary of $a$ lies in $\bdry \Sigma$.  Further assume that each $a_i$ separates a component of $\Sigma$.  Let $G\colon\Sigma\looparrowright S^3$ be an immersion with transverse self-intersections identifying each $a_i$ with $b_i$ and which is otherwise an embedding.  Then we say that $G$ is a \emph{fusion surface} with \emph{fusion arcs} $a$ and $b$.\end{definition}
  
  The prior arguments imply that if a link is a fusion of a boundary link then it bounds a fusion surface.  Any intersection between a fusion band and a Seifert surface results in a fusion arc.    The converse can also be easily seen - the boundary of $\Sigma\setminus N(a)$ is a boundary link and each fusion arc in $a$ corresponds to a fusion band.

Let $G\colon\Sigma\immerse S^3$ be a fusion surface with fusion arcs $a$ and $b$.  The \emph{Seifert form} $V_G\colon  H_1(\Sigma)\times H_1(\Sigma)\to \Z$ is given as follows.  Note that since each $a_i$ separates, we may represent each element in $H_1(\Sigma)$ by a sum of simple closed curves in $\Sigma$ each of which is disjoint from each $a_i$ and $b_i$. Let $\alpha, \beta$ be some such simple closed curves. Then we define $V_G([\alpha],[\beta]) := \lk(G(\alpha), G^+(\beta))$, where $G^+(\beta)$ is obtained by pushing $G(\beta)$ off of $G(\Sigma)$ in the positive normal direction.  We extend this linearly to obtain the Seifert form $V_G\colon H_1(\Sigma)\times H_1(\Sigma) \to \Z$.  We make no attempt here to study the degree to which the Seifert form is independent of $G$ and is an invariant of $L$. 

We observe in Lemma~\ref{lem: S-equiv to bdry} below that given any fusion surface, there is an embedded surface with the same Seifert form. We say that two fusion surfaces $G\colon\Sigma\immerse S^3$ and $G'\colon\Sigma'\immerse S^3$ have \emph{isomorphic Seifert forms} if there is an orientation preserving homeomorphism $\phi\colon\Sigma\to \Sigma'$ so that  $V_G(\alpha, \beta) = V_{G'}(\phi_*(\alpha),\phi_*(\beta))$ for every $\alpha, \beta\in H_1(\Sigma)$.

\begin{lemma}\label{lem: S-equiv to bdry}
If $G\colon\Sigma\immerse S^3$ is a fusion surface, then there is an embedding $G'\colon\Sigma'\into S^3$ such that $G$ and $G'$ give isomorphic Seifert forms.
\end{lemma}
\begin{proof}Let $a=a_1\cup \dots \cup a_\ell$ and $b = b_1\cup \dots \cup b_\ell$ be fusion arcs of $G$.   Let $c=c_1\cup \dots \cup c_\ell\subseteq \Sigma$ be disjoint embedded arcs with $c_i$  running from a point on $\bdry \Sigma$ to a point in $b_i$ and which are otherwise disjoint from $a$ and $b$.  There is an isotopy $\phi\colon\Sigma\times[0,1]\to \Sigma$ supported in a small neighborhood of $b\cup c$ with $\phi_0 = \Id_G$ and $\phi_1(\Sigma)$ disjoint from $b \cup c$.  By setting $G' := G \circ \phi_1 \colon \Sigma \hookrightarrow S^3$, it is clear that $G'$ is an embedding. See Figure~\ref{fig:removeIsectn} for an illustration. Moreover, any classes in $H_1(\Sigma)$ can be represented by disjoint unions of embedded loops away from a neighborhood of $b \cup c$. Thus, $G'$ and $G$ have isomorphic Seifert forms and $G'$ is an embedding, completing the proof. \end{proof}

\begin{figure}[h]
\begin{picture}(270,110)
\put(0,0){\includegraphics[height=.15\textheight]{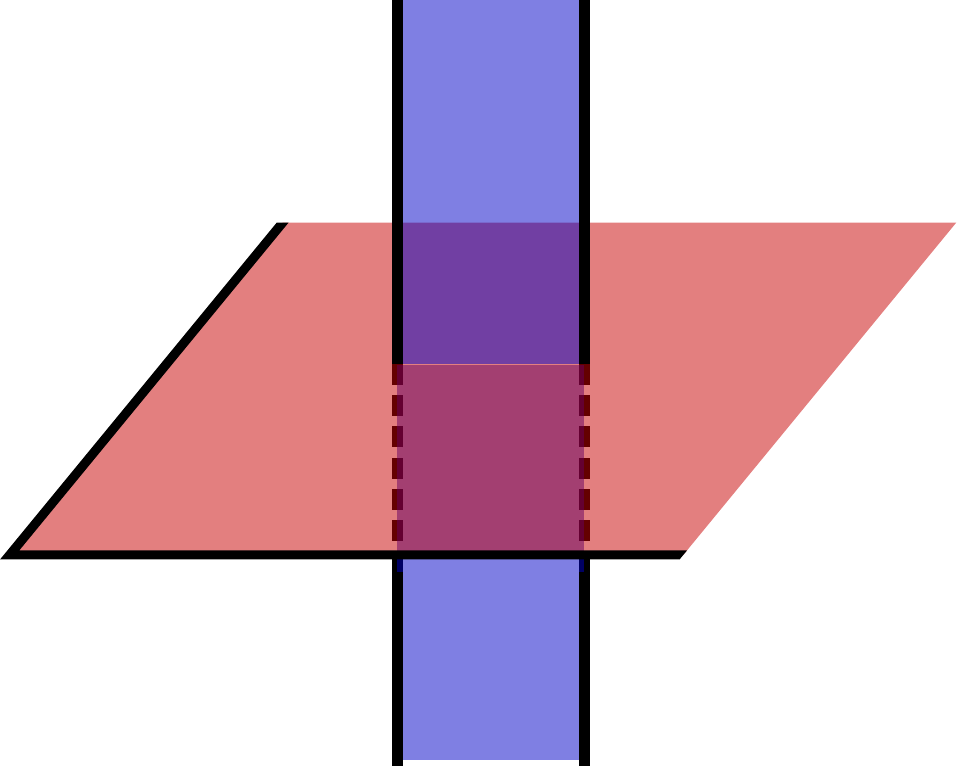}}
\put(130,0){\includegraphics[height=.15\textheight]{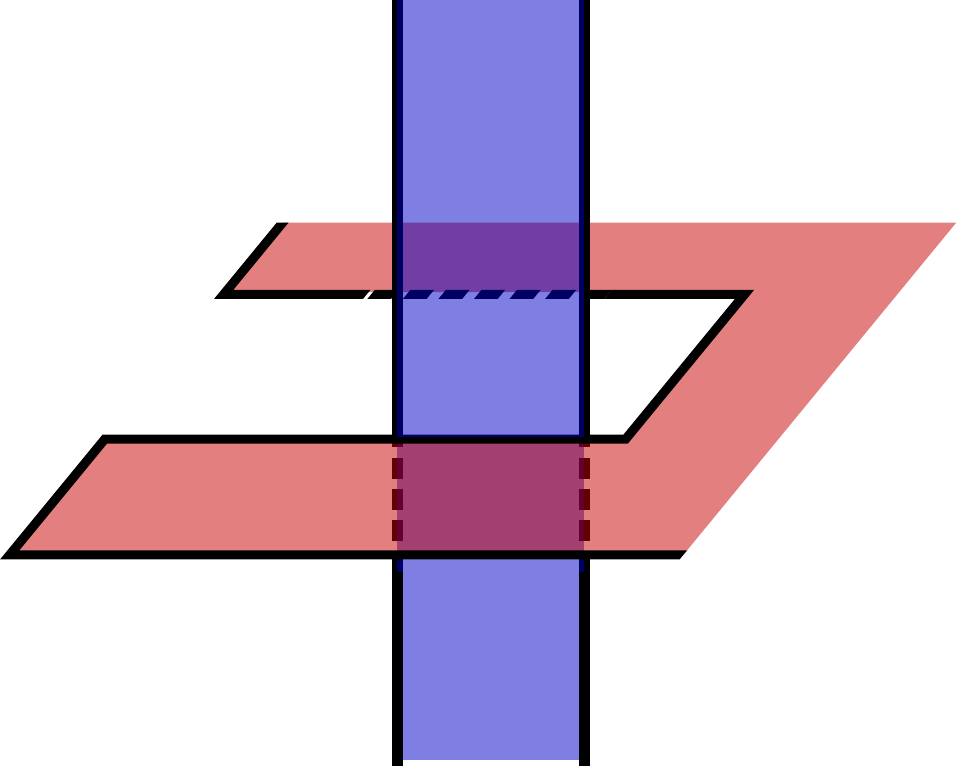}}

\end{picture}\vspace{.3cm}
\caption{Left: An intersection arc of a fusion complex.  Right: Removing that intersection.  }\label{fig:removeIsectn}
\end{figure}

\subsection{Fusion complexes, metabolic linking forms, and $0.5$-solvability}\label{metabolic to 0.5 solvable}
In Section~\ref{sect: C-cplx}, we presented a sufficient condition for a link to be 0.5-solvable using the linking form on a C-complex. We give a similar condition using the linking form on a fusion complex.

\begin{definition}\label{defn: metabolic form}
Let $G\colon\Sigma\immerse S^3$ be a fusion surface with fusion arcs $a=a_1\cup \dots \cup a_\ell$ and $b=b_1\cup \dots \cup b_\ell$.  We say that $G$ admits a \emph{metabolic linking form} if

\begin{itemize}
\item There is a  collection of embedded simple closed curves $\{\alpha_1, \dots, \alpha_g, \beta_1, \dots, \beta_g\}$ all disjoint except that $\alpha_i$ intersects $\beta_i$ transversely and positively in a single point, where $g=g(\Sigma)$ is the sum of the genera of the components of $\Sigma$.  This set of curves is called a symplectic basis for $H_1(\Sigma)$.
\item $V_{G}([\alpha_i], [\alpha_j]) = 0$ for all $i,j$.
\item $\beta_i\cap a_j = \beta_i\cap b_j =\emptyset$ for all $i,j$.
\end{itemize}
\end{definition}

The existence of a metabolic linking form implies $0.5$-solvability as we prove now by passing to the setting of C-complexes.  

 \begin{proposition}\label{prop: metabolic implies 0.5 solvable}
A fusion of a boundary link which admits a fusion surface with metabolic linking form is $0.5$-solvable.  
 \end{proposition}
 
 \begin{proof}
 Let $G\colon\Sigma\immerse S^3$ be a fusion surface with fusion arcs $a=a_1\cup \dots \cup a_m$ and $b=b_1\cup \dots \cup b_m$.  Let $\{\alpha_1, \dots, \alpha_g, \beta_1, \dots, \beta_g\}$ be a collection of embedded simple closed curves from Definition~\ref{defn: metabolic form}. 
Since $b_1,\dots, b_m$ is a disjoint collection of arcs interior to $\Sigma$, we can arrange that $\alpha_i\cap b_j =\emptyset$ for all $i, j$. Moreover, since each $a_j$ separates $\Sigma$, the algebraic intersection $\alpha_i\cdot a_j=0 $ for all $i,j$.

Our proof transforms $G$ into a C-complex with metabolic linking form by using the moves of Figure~\ref{fig:removefusions} near each fusion arc.  We explain these moves in detail in the coming proof, but we lead with an overview. If a pair of fusion arcs $(a_i, b_i)$ each sit in the same component of $\Sigma$, then we will use the move of Figure~\ref{fig:ResolveSelfFusion}.  This will increase the genus of $\Sigma$ by 1.  If the pair $a_i$ and $b_i$ sit in different components if $\Sigma$ then the move of Figure~\ref{fig:ResolveFusion} replaces the ribbon intersection at $G(a_i)$ with two clasps.  Once we perform these moves at every pair of intersection arcs, we will have constructed a C-complex.  A detailed book-keeping will reveal that this C-complex admits a metabolic linking form.  An appeal to Theorem~\ref{thm: metabolic C-cplx implies 0.5-sol} will complete the proof.

Up to reordering, we assume that for $1\le i\le q$, the $i$'th pair of fusion arcs, $a_i$ and $b_i$ sit in the same component of $\Sigma$.  Let $U_i\subseteq \Sigma$ be a small open disk neighborhood of $b_i$ and $a_i^+, a_i^-\subseteq \Sigma$ be positive and negative pushoffs of $a_i$.  Remove $U_i$ and the band between $a_i^+$ and $a_i^-$ from $\Sigma$.  The resulting surface $\Sigma^0$ has three more boundary components than does $\Sigma$.  Define $\Sigma'$ by adding to $\Sigma^0$ a pair of bands from $a_i^+$ and $a_i^-$ to $\bdry U_i$. First restrict $G$ to $\Sigma_0$ and then extend it over these bands as in Figure~\ref{fig:ResolveSelfFusion} to get $G'\colon\Sigma'\immerse S^3$.  This operation increases the genus by $1$.  Let $\alpha_i'$ be the result of pushing $\bdry U$ into  the interior of $G'$ and $\beta_i'$ be an intersection dual to $\alpha_i'$ disjoint from all of $\alpha_1, \beta_1, \dots, \alpha_g, \beta_g$.  Perform this move at every pair of fusion arcs $a_i, b_i$ contained in a single component of $\Sigma$.

For $1\le i\le m-q$, $a_{i+q}\subseteq \Sigma_j$ and $b_{i+q}\subseteq \Sigma_k$ lie in different components of $\Sigma$.  In this case we modify $G$ by pushing along an arc from $\bdry \Sigma$ to $b_i$.  This move is depicted in Figure~\ref{fig:ResolveFusion} and produces two new clasps $c_i^+$ and $c_i^-$ of opposite sign together with a loop $\ell_i$ which passes once over $c_i^+$ and once over $c_i^-$.  Notice that $\ell_i$ separates $G(\Sigma_k)$ but may fail to separate $G(\Sigma_j)$.  Perform this move at each set of fusion arcs contained in different components of $\Sigma$.

\vspace{.8cm}
  \begin{figure}[h]
\begin{subfigure}{.4\textwidth}
\begin{picture}(190,65)
\put(0,0){\includegraphics[height=.1\textheight]{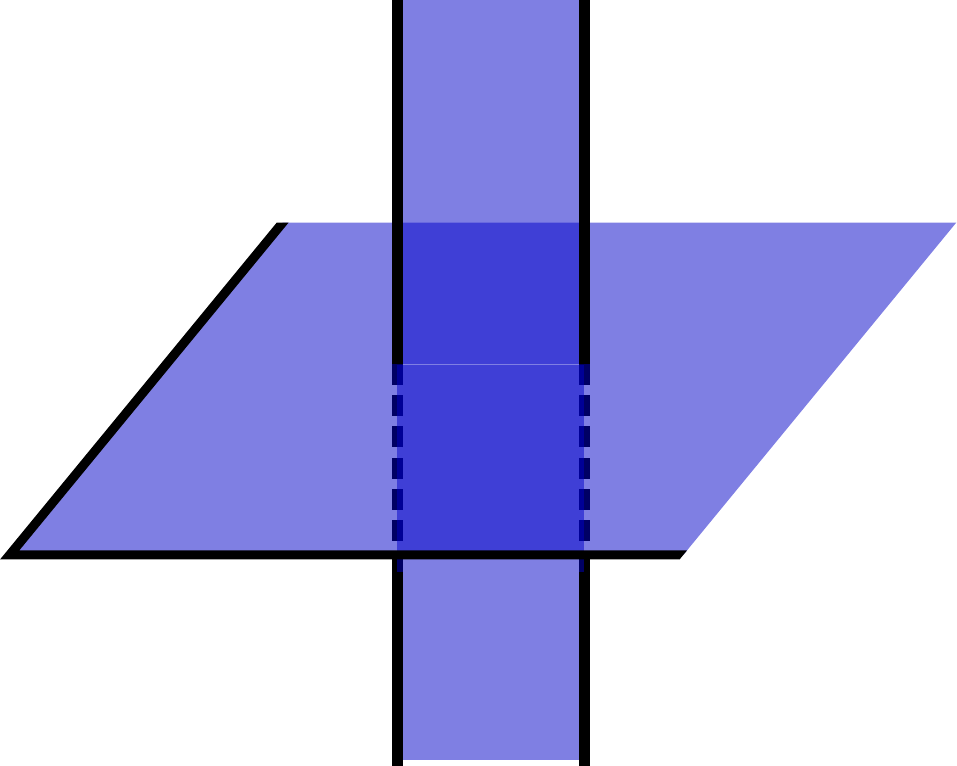}}
\put(80,35){{$\mapsto$}}
\put(100,0){\includegraphics[height=.1\textheight]{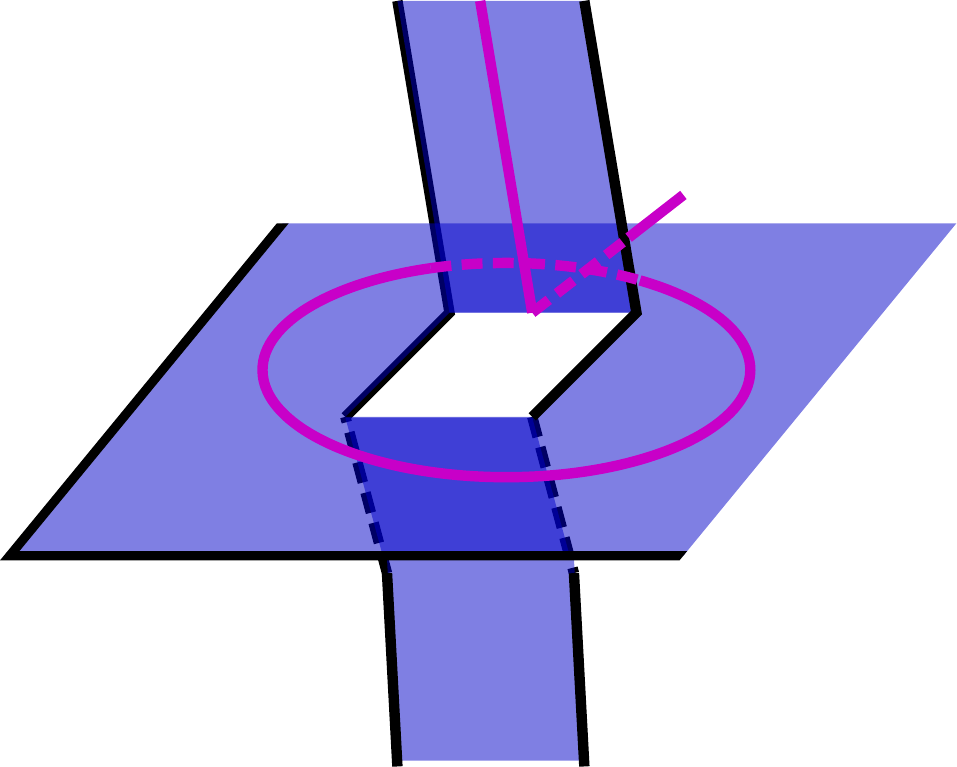}}
\put(120,25){$\alpha'$}
\put(160,50){$\beta'$}
\end{picture}
\caption{}\label{fig:ResolveSelfFusion}
\end{subfigure}
%
\begin{subfigure}{.4\textwidth}
\begin{picture}(190,65)
\put(0,0){\includegraphics[height=.1\textheight]{fusionIntersectionArc.pdf}}
\put(80,35){{$\mapsto$}}
\put(100,0){\includegraphics[height=.1\textheight]{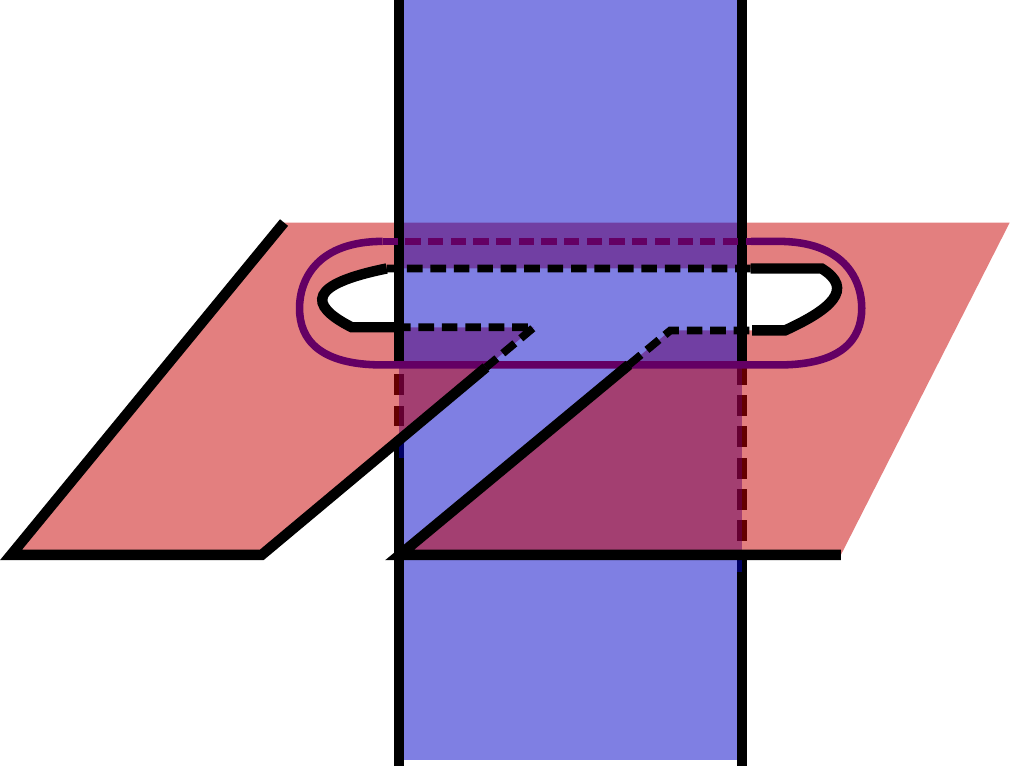}}
\put(170,50){$\ell$}
\end{picture}
\caption{}\label{fig:ResolveFusion}
\end{subfigure}
\caption{Left: A fusion self-intersection is resolved by increasing genus by 1.  Right:  A fusion intersection between different components replaced by two clasps.}\label{fig:removefusions}
\end{figure}

Let $G'\colon\Sigma'\immerse S^3$ be the resulting immersion.  Then $F = G'(\Sigma')$ is a C-complex.   We claim that this C-complex has metabolic linking form.  The loops $\alpha_1,\dots, \alpha_g, \alpha_1'\dots, \alpha_q'$ and $\ell_{1},\dots, \ell_{m-q}$ come close to satisfying conditions in Definition~\ref{defn: metabolic}.
\begin{itemize}
\item Each $\ell_i$ passes once over two clasps of opposite signs.
\item Each clasp sits in exactly one $\ell_i$.
\item $\alpha_1,\dots, \alpha_g, \alpha_1'\dots, \alpha_q'$ forms a non-separating collection of disjoint simple closed curves where $g+q$ is the genus of $F$.
\item $\alpha_i'$ is disjoint from each $\ell_k$.
\item For each $\epsilon\in \{\pm1\}^n$, the linking form $V^\epsilon$ vanishes on the span of $\{[\alpha_i],[\alpha_j'], [\ell_k]\}_{i,j,k} \subseteq H_1(F)$.
\end{itemize}

It remains to modify the various $\ell_i$ so that:
\begin{itemize}
\item For each $i$, if $\ell_i\subseteq F_j\cup F_k$, then $\ell_i$ separates $F_j$ and $F_k$.
\item $\ell_i\cap \alpha_j=\emptyset$ for all $i, j$.
\end{itemize}

By construction $\ell_j$ is disjoint from each $\beta_i$ and each $\alpha_i'$.  While $\ell_j$ may not be disjoint from some $\alpha_i$, the fact that each the fusion arc $a_j$ separates $\Sigma$ implies that $\ell_j$ has zero algebraic intersection with $\alpha_i$. On the other hand, the algebraic intersection $\ell_j\cdot \beta_i'$ may fail to vanish.

Fix one of the curves $\alpha_i$.  For each time that $\alpha_i$ intersects some $\ell_j$, modify $\ell_j$ by a band sum with $\beta_i$ as in Figure~\ref{fig:removeLoopCrossing}.  
{This replacement reduces $\#\ell_j \cap \alpha_i$ by 1, creates no new intersections, and replaces $[\ell_k]\in H_1(F)$ with $[\ell_k]\pm [\beta_i]$, depending on the sign of intersection being removed.}
  Since the algebraic intersection $\alpha_i\cdot \ell_k = 0$, after we remove every such intersection, the homology class of $\ell_k$ is unchanged.  

   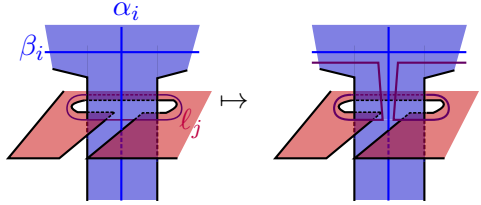
\begin{figure}[h]
\begin{picture}(200,85)
\put(0,0){\includegraphics[height=.1\textheight]{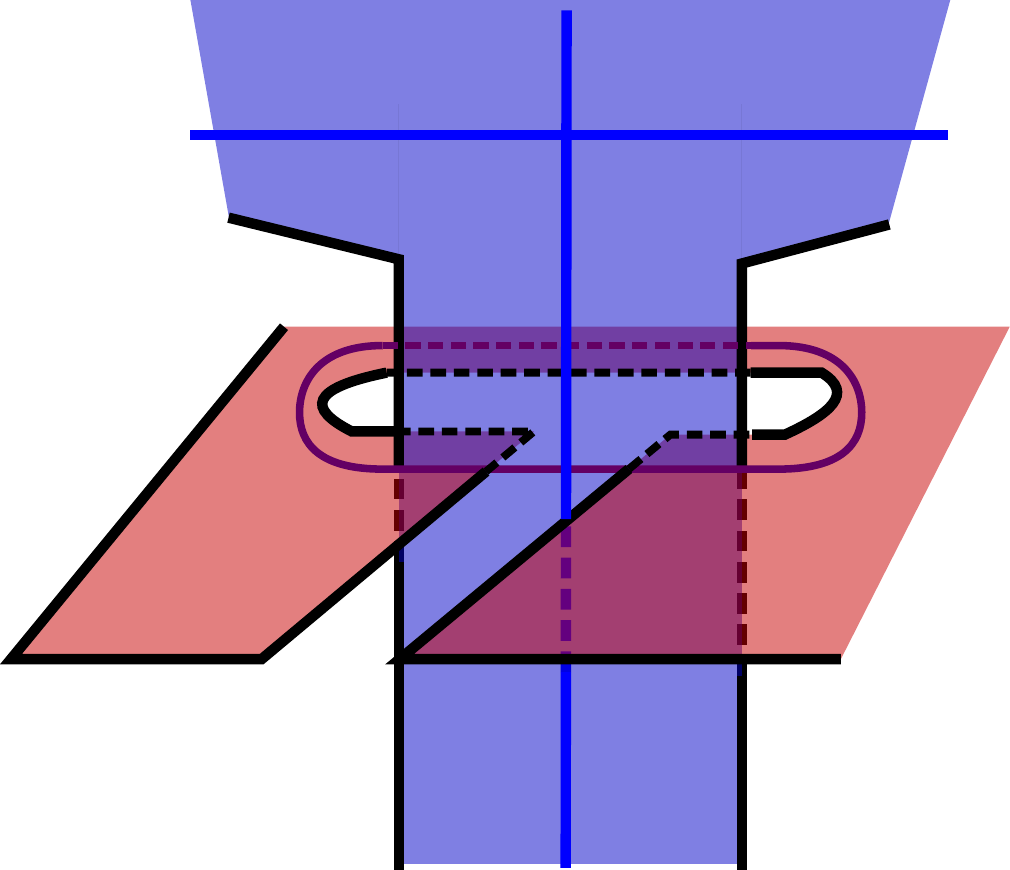}}
\put(5,55){\textcolor{blue}{$\beta_i$}}
\put(40,70){\textcolor{blue}{$\alpha_i$}}
\put(65,27){\textcolor{purple}{$\ell_j$}}
\put(80,35){{$\mapsto$}}
\put(100,0){\includegraphics[height=.1\textheight]{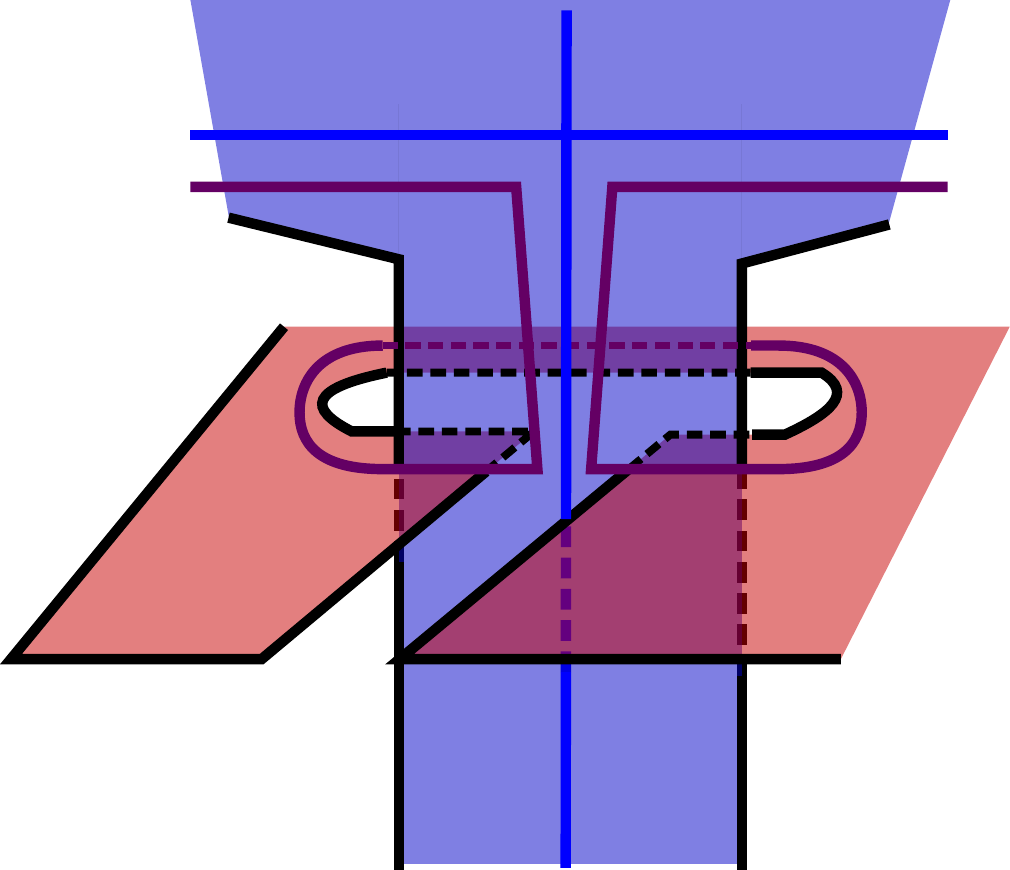}}
\end{picture}\vspace{.5cm}
\caption{Right:  A loop $\ell\subseteq F_i\cap F_j$ intersects a loop $\alpha_i\subseteq F_i$ with intersection dual $\beta_i\subseteq F_i$.  Right: sliding $\ell$ over $\beta_i$ reduces $\#\ell \cap \alpha_i$ by 1.  }\label{fig:removeLoopCrossing}
\end{figure} 

At this stage each $\ell_j$ may still intersect the intersection duals $\beta_i'$ for the curves $\alpha_i'$.  We follow a similar logic, for each intersection between $\beta_i'$ and $\ell_j$ we modify $\ell_j$ by a band sum with $\alpha_i'$ to remove this intersection.  This move transforms $[\ell_j]\in H_1(F)$, into $[\ell_j]\pm[\alpha_i']$. Since $V^\epsilon$ already vanished on the span of $\{[\alpha_i],[\alpha_j'], [\ell_k]\}_{i,j,k} \subseteq H_1(F)$, it still vanishes after this replacement is made.  

We now have that $\ell_j\cap F_i$ is disjoint from each curve that forms a symplectic basis for $H_1(F_i)$.  It follows that $[\ell_k\cap F_i]=0$ in $H_1(F_i, \bdry F_i)$ and so $\ell_k\cap F_i$ separates, as we needed. Hence, the C-complex $F$ has metabolic linking form, and by Theorem~\ref{thm: metabolic C-cplx implies 0.5-sol}, we conclude that 
$L$ is $0.5$-solvable.\end{proof}

 \subsection{Links with isomorphic Seifert forms are $0.5$-solvably equivalent.}\label{metabolic fusion}

In Lemma~\ref{lem: S-equiv to bdry}, we saw that every fusion of a boundary link has a Seifert form that is isomorphic to the Seifert form of a boundary link.  Here, we prove that fusions of boundary links with isomorphic Seifert forms are $0.5$-solve-equivalent. By combining the two, we will prove Theorem~\ref{thm: HBL = BL} at the end of this section.

\begin{theorem}\label{thm: S-equiv implies 0.5-sol}
If $L$ and $L'$ are fusions of boundary links with fusion surfaces $G$ and $G'$ that have isomorphic Seifert forms, then $L$ and $L'$ are $0.5$-solve-equivalent. \end{theorem}

\begin{proof}By Lemma~\ref{lem: S-equiv to bdry} there is an embedding $G'':\Sigma''\into S^3$ so that $G''$ has Seifert form isomorphic to the Seifert form of $G$.  As a consequence it isomorphic to the Seifert form of $G'$.  Since $0.5$-solve-equivalence is an equivalence relation, we see that it suffices to prove the theorem in the case that $L'$ is a boundary link and $G'$ is an embedding.  

Let $L$ be a fusion boundary link with fusion complex $G:\Sigma\immerse S^3$.  Let $L'$ be a boundary link with embedded Seifert surface $G':\Sigma'\into S^3$.  Assume further that $G$ and $G'$ have isomorphic Seifert forms, so that there is an orientation preserving homeomorphism $\phi\colon\Sigma \to \Sigma'$ so that $V_G(\alpha, \beta) = V_{G'}(\phi_*(\alpha), \phi_*(\beta))$ for all $\alpha, \beta \in H_1(\Sigma)$.  Let $\{a_i, b_i\}$ be the fusion arcs of $G$.  

Recall that for any link $J$,  $-J$ is the reverse of the mirror image of $J$.  A fusion complex $H\colon\Sigma\immerse S^3$ for $J$ gives rise to a fusion complex $-H$ for $-J$ as we presently describe. Let $-\Sigma$ be the orientation reverse of $\Sigma$, $r\colon -\Sigma\to \Sigma$ be the identity map and $m\colon S^3\to S^3$ be an orientation reversing homeomorphism.  $-J$ admits a fusion complex denoted $-H$ given by $-H:= m\circ H\circ r\colon-\Sigma\to S^3$.  
For any $\alpha \in H_1(-\Sigma)$, 
$$
(-H)^+(\alpha) = (m\circ H\circ r)^+(\alpha) = (m\circ H)^-(\alpha)=m(H^+(\alpha))
$$
where the second equality follows since the positive normal to $H\circ r$ is the negative normal to $H$, and the third follows since $m:S^3\to S^3$ sends the positive normal vector to $H$ to the negative normal vector to $m(H)$.  Similarly for any $\beta\in H_1(-\Sigma)$, 
$$
(-H)(\beta) = (m\circ H\circ r)(\beta) = (m\circ H)(\beta)=m(H(\beta)).
$$
Finally, since $m:S^3\to S^3$ is orientation reversing, 
$$V_{-H}(\alpha, \beta) = \lk(m(H^+(\alpha)), m(H(\beta))) = -\lk(H^+(\alpha), H(\beta)) = -V_H(\alpha, \beta)$$

Let $\delta_1, \dots, \delta_n$ be bands for an exterior band sum of $L$ with $-L'$.  We further require that these bands be disjoint from the interiors of $G$ and $G'$.  A fusion complex  for $L\#_{\delta}-L'$ can now be obtained by starting with $G\cup G'$ and adding bands $\delta_1,\dots, \delta_n$ to it. The resulting surface is denoted by $G\#_\delta -G'\colon \Sigma\# -\Sigma \immerse S^3$.  

Let $\{s_i, t_i\}_{1\leq i \leq g}$ be a collection of embedded simple closed curves in $\Sigma$ all disjoint except that for each $i$, $s_i$ intersects $t_i$ transversely in a single point and so that their homology classes form a symplectic basis for $H_1(\Sigma)$.
Since the fusion arcs $a_1, a_2,\dots a_k$ each separate and are pairwise disjoint from each other, we can assume that for all $i$ and $j$, the closed curves $\{s_i, t_i\}_{1\leq i \leq g}$ are all disjoint from all the fusion arcs in $\Sigma$.    It follows that the homology classes of $\{s_i, t_i, \phi(s_i),\phi(t_i)\}_{1\leq i \leq g}$ forms a symplectic basis for $\Sigma\# -\Sigma$. With respect to this basis 
$$
V_{G\#_\delta -G'} = V_{G}\oplus V_{-G'} =  V_{G}\oplus -V_{G}
$$

As in Figure~\ref{fig:get Metab basis}, we may band these together to get a collection of $2g$ nonseparating disjoint simple closed curves 
$$\{s_i -\phi(s_i), t_i -\phi(t_i)\}_{1\leq i \leq g}$$
in the genus $2g$ surface $\Sigma\# -\Sigma$. These have intersection duals $\{t_i, \phi(s_i)\}_{1\leq i \leq g}$, in the sense that 
$$\{s_i -\phi(s_i), t_i, t_i -\phi(t_i), \phi(s_i)\}_{1\leq i \leq g}$$
are all disjoint except that for all $i$, $s_i-\phi(s_i)$ intersects $t_i$ and $t_i -\phi(t_i)$ intersects $s_i$ transversely in a single point.  Additionally, as we arranged that the fusion arcs $\{a_i\}$ for $G$ are disjoint from $s_i$ and $t_i$, it follows that $s_i$ and $t_i$ are disjoint from the fusion arcs of $G\#-G'$.
By direct computation, 
$$V_{G\#_\delta -G'}(s_i -\phi(s_i), t_j -\phi(t_j)) = V_{G}(s_i, t_j) + V_{-G'}(-\phi(s_i), -\phi(t_j)) = V_{G}(s_i, t_j) - V_{G}(s_i, t_j)=0.$$
The argument that $V_{G\#_\delta -G'}$ vanishes on the rest of $\{s_i +\phi(s_i), t_i + \phi(t_i)\}_{1 \leq i \leq g}$ is identical.  
 We conclude that $V_{G\#_\delta -G'}$ satisfies Definition~\ref{defn: metabolic form}.  Finally, by Theorem~\ref{prop: metabolic implies 0.5 solvable} we have that $L\#_\delta -L'$ is $0.5$-solvable, and by Proposition~\ref{prop: solvability and interior sum} $L$ is $0.5$-solve-equivalent to $L'$.\end{proof}

\vspace{.8cm}
  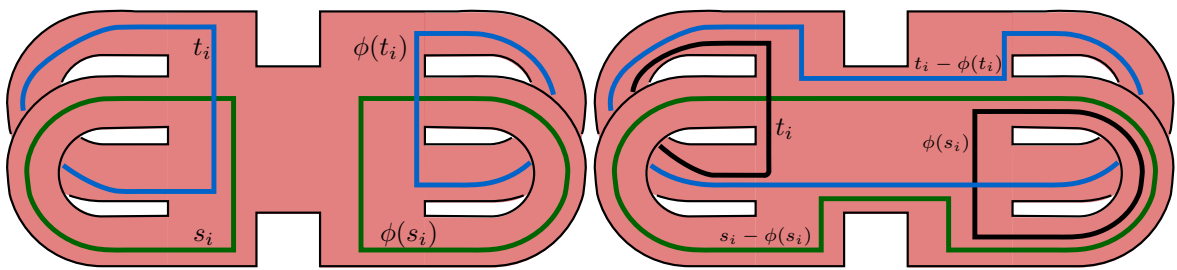
\begin{figure}[h]
\begin{picture}(420,100)
\put(0,0){\includegraphics[angle = 90, width=.48\textwidth]{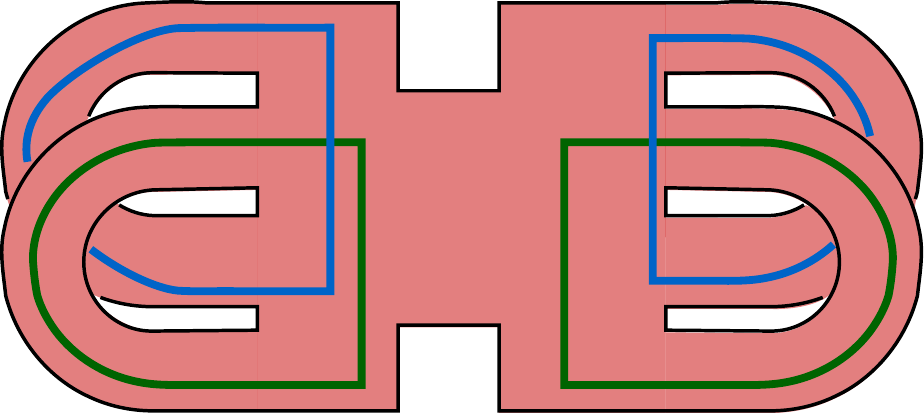}}
\put(70,10){\small{$s_i$}}
\put(140,10){\small{$\phi(s_i)$}}
\put(70, 80){\small{$t_i$}}
\put(130,80){\small{$\phi(t_i)$}}

\put(220,0){\includegraphics[angle = 90, width=.48\textwidth]{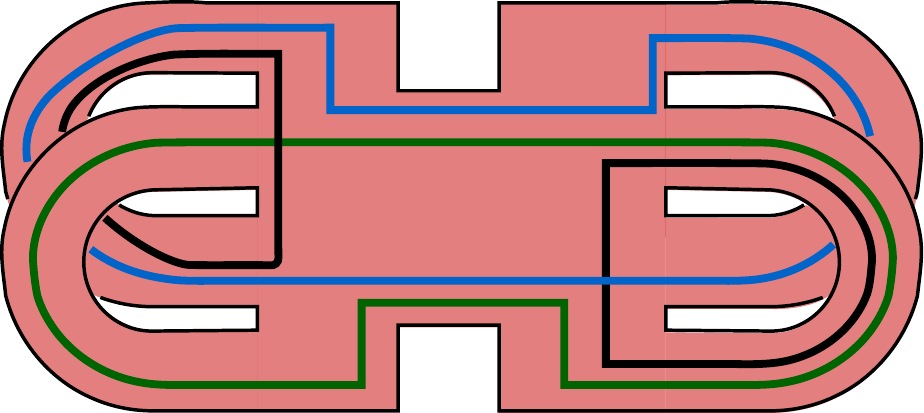}}

\put(267,10){\tiny{$s_i-\phi(s_i)$}}
\put(288,50){\small{$t_i$}}
\put(340,75){\tiny{$t_i-\phi(t_i)$}}
\put(343,45){\tiny{$\phi(s_i)$}}

\end{picture}\vspace{.5cm}
\caption{Left: Some curves in a symplectic basis for $\Sigma\# -\Sigma$.  Right: a new symplectic basis for $\Sigma\#_\delta -\Sigma'$. }\label{fig:get Metab basis}
\end{figure}

We are now ready for a very short proof of Theorem~\ref{thm: HBL = BL}, whose statement we recall.

\begin{reptheorem}{thm: HBL = BL}
Any sublink of a homology boundary link is $0.5$-solve-equivalent to a boundary link. 
\end{reptheorem}

\begin{proof}
Let $L$ be a sublink of a homology boundary link $L^1$.  By \cite[Corollary 4.1]{CocLev91} $L^1$ is concordant to a fusion of a boundary link, $L^2$.  Let $G$ be a fusion surface for $L^2$. By Lemma~\ref{lem: S-equiv to bdry} there is a boundary link $L^3$ with an embedded Seifert surface $G'\colon\Sigma'\into S^3$ so that $G$ and $G'$ have isomorphic Seifert forms. Finally, by Theorem~\ref{thm: S-equiv implies 0.5-sol}  $L^2$ is $0.5$-solve-equivalent to $L^3$.  

Thus, $L$ is a sublink of a link that is $0.5$-solve-equivalent to a boundary link.  As sublinks of $0.5$-solve-equivalent links are $0.5$-solve-equivalent, and a sublink of a boundary link is again a boundary link, this completes the proof.  \end{proof}

\section{Iterated satellite operations and highly solvable knots}\label{sect:High sol background}

For the remainder of the paper, we shift our focus away from $0.5$-solve-equivalence to the comparison of homology boundary links and boundary links deep in the solvable filtration.  In this section, we recall a construction which is used  to produce links that sit in $\F_n$ for $n$ large but which often do not sit in the subsequent term $\F_{n.5}$.  

A \emph{generalized doubling operator}  \cite{CHL4} $R_{\eta}$ consists of a knot $R$ and a link $\eta$ in $S^3 \sm \nu R$ that forms a trivial link in $S^3$ and whose every component is nullhomologous in the exterior of $R$.  When $\eta$ has only one component, $R_\eta$ is a \emph{doubling operator}.  
Given a knot $J$,  the \emph{result of the  generalized doubling operation} $R_{\eta}(J)$ is obtained by tying the strands of $R$ passing through each of the disks bounded by the components of $\eta$ into $J$. An example is depicted in Figure~\ref{fig: doubling}.  The application of generalized  doubling operators is a standard strategy used to construct knots and links deep in the solvable filtration.

\begin{figure}[h]

\begin{subfigure}{.3\textwidth}
         \centering
\begin{picture}(90,110)
\put(-5,5){\includegraphics[height=.15\textheight]{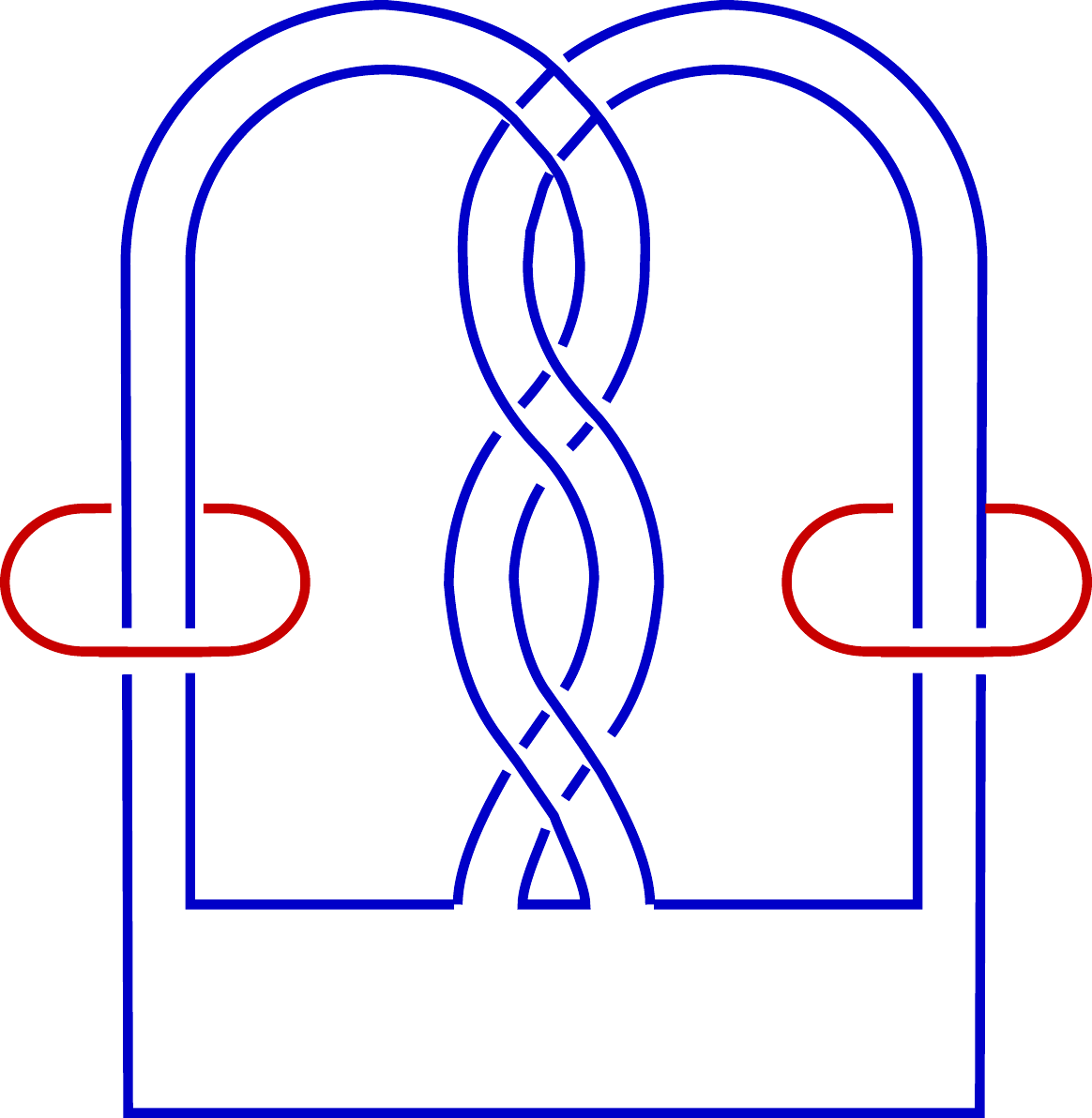}}
\put(40,10){$R$}
\put(-5,64){$\eta_1$}
\put(85,64){$\eta_2$}
\end{picture}
\caption{}
\end{subfigure}
\begin{subfigure}{.3\textwidth}
         \centering
\begin{picture}(100,100)
\put(0,0){\includegraphics[height=.15\textheight]{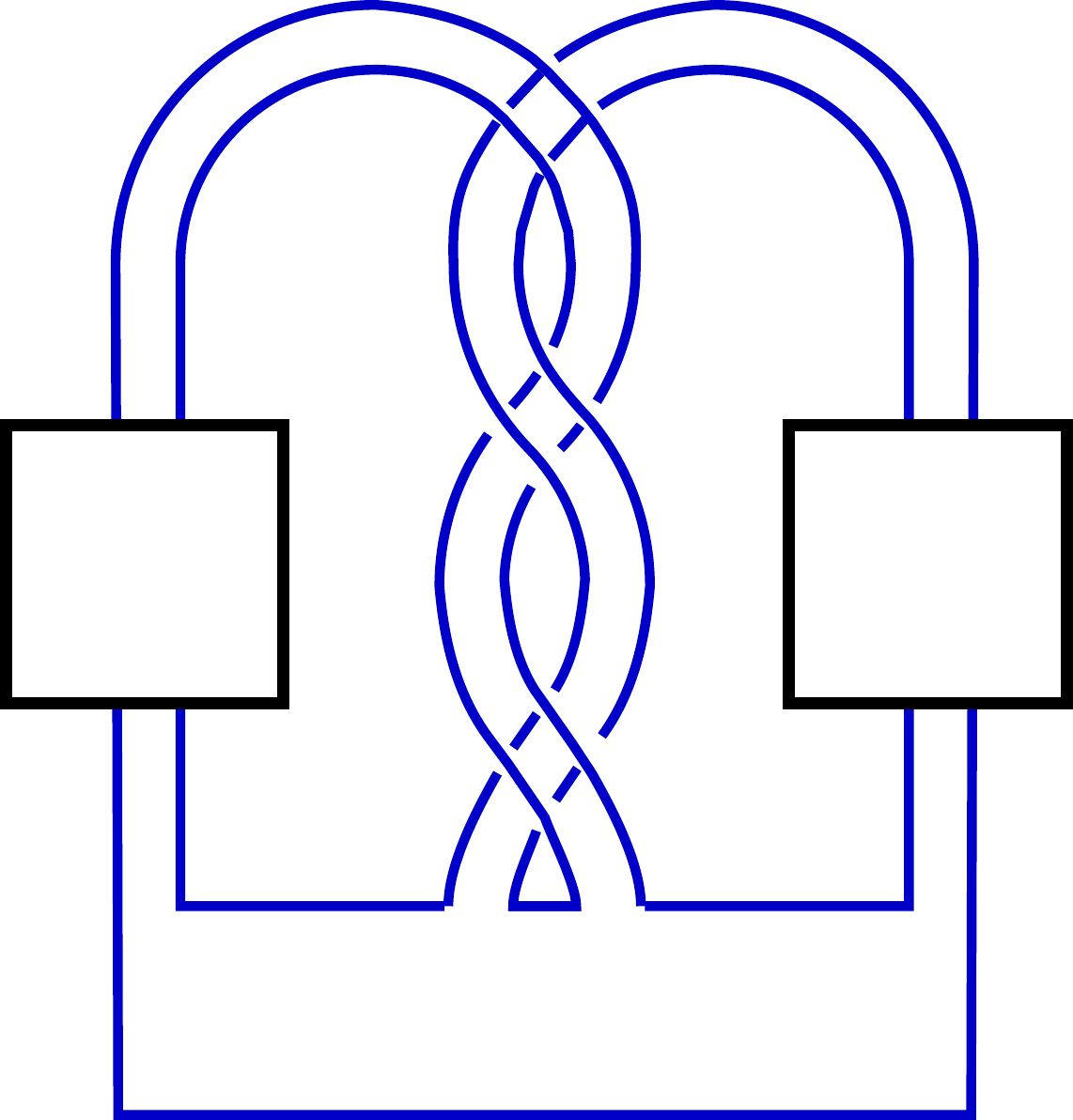}}
\put(30,7){$R_{\eta_1,\eta_2}(J)$}
\put(8.3,45){$J$}
\put(79,45){$J$}
\end{picture}
\caption{}
\end{subfigure}
\caption{Left: A generalized doubling operator $R_{\eta_1, \eta_2}$.  Right:  The result of the  generalized  doubling operation $R_{\eta_1,\eta_2}(J)$.}\label{fig: doubling}
\end{figure}
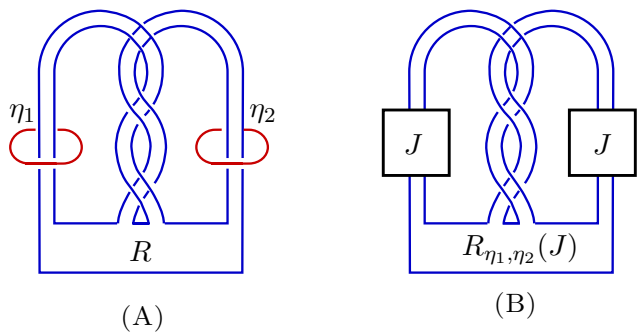 

\begin{proposition}\label{prop: increase sol}Let $n, p\in \frac{1}{2}\N_0$ and $m, k\in \N_0$.  Suppose $R_{\eta}$ is a generalized doubling operator where $R$ is an $n$-solvable knot with $n$-solution $W$, and the homotopy classes of the components of $\eta=\eta_1\cup\dots \cup\eta_m$ lie in $\pi_1(W)^{(k)}$.  If $J$ is a $p$-solvable knot, then $R_\eta(J)$ is $\min(n, p+k)$-solvable.  
\end{proposition}
\begin{proof}
Similar results appear elsewhere in the literature and have identical proofs, for example \cite[Proposition 3.1]{CocTei04} and  \cite[Proposition 2.7]{CHL2011}. The precise statement above does not seem to appear.  As such we will provide a short argument in the case the $n,p\in \N_0$. The argument when either or both of $n$ and $p$ are half-integers is identical.  Let $W$ be an $n$-solution for $R$ and $X_i, Y_i\subseteq W$ be the $n$-Lagrangians and $n$-duals in $W$.  Let $U_1,\dots, U_m$ be copies of a $p$-solution $U$ for $J$, where $m$ is the number of components of $\eta$, and let $A_i, B_i\subseteq U$ be $p$-Lagrangians and $p$-duals.

Let $Z$ be the 4-manifold obtained by starting with the disjoint union $W\sqcup U_1\sqcup\dots\sqcup U_m$ and then identifying a neighborhood of the meridian of $J$ in $M_J = \bdry U_i$ with a neighborhood of $\eta_i$ in $M_R=\bdry W$ so that $0$-framed longitudes are identified with meridians.  We assert that $Z$ is an $n'$-solution for $R_{\eta}(J)$ where $n' = \min(n, p+k)$.  The fact that $\bdry Z\cong M_{R_{\eta}(J)}$ follows since the meridian of $J_i$ in $M_{J_i}$ is isotopic to the core of the solid torus added to $S^3 \sm J_i$ to produce the zero surgery.  The facts that $H_1(Z)$ is freely generated by the meridians of the components of $R_{\eta}(J)$ and that $H_2(Z)\cong H_2(W)\oplus \Oplus_j H_2(U_j)$ follow from a Mayer-Veitoris argument exactly as in \cite[Proposition 2.7]{CHL2011}.  Thus, the homology classes of the Lagrangians and duals in $W$ together with the Lagrangians and duals in the various $U_i$ give a basis for $H_2(Z)$.  

All that remains is to check that condition \pref{lift basis} of Definition~\ref{defn:solvable} holds.  For any Lagrangian $X_i$ of $W$, $\pi_1(X_i)\to \pi_1(Z)$ factors though $\pi_1(W)$.  As the image of $\pi_1(X_i)\to \pi_1(W)$ is contained in $\pi_1(W)^{(n)}$, it follows that $\im(\pi_1(X_i)\to \pi_1(Z))\subseteq \pi_1(Z)^{(n)}\subseteq \pi_1(Z)^{(n')}$.  The same argument applies to the duals.  

The meridian of $J$ in $U_j$ is identified with $\eta_j$ which sits in $\pi_1(Z)^{(k)}$.  This implies that $\im(\pi_1(U_j)\to \pi_1(Z))\subseteq \pi_1(Z)^{(k)}$. (For the details in this group theoretic claim we refer the reader to the proof of \cite[Proposition 2.7]{CHL2011}.)   Now consider any Lagrangian $A_i$ of $U_j$.  Since $\im(\pi_1(A_i)\to \pi_1(U_j))\subseteq \pi_1(U_j)^{(p)}$, it follows that
$$\im\left(\pi_1(A_i)\to \pi_1(U_j)\to \pi_1(Z)\right)\subseteq 
\im\left(\pi_1(U_j)\to \pi_1(Z)\right)^{(p)}\subseteq \pi_1(Z)^{(k+p)}\subseteq \pi_1(Z)^{(n')}.$$
Again, the same argument applies to the dual $B_i$. Thus, the surfaces $X_i$, $Y_i$, $A_i$, and $B_i$ together form $n'$-Lagrangians and $n'$-duals for $Z$, completing the proof.
\end{proof}

Iterated application of  generalized doubling operators produces highly solvable knots.   Indeed, let $J$ be a $0$-solvable knot and $R$ be a slice knot, so that $R$ is $n$-solvable for every $n\in \N$.  Let $R_\eta$ be a generalized doubling operator.  By the definition of a generalized doubling operator, each component of $\eta$ sits in $\pi_1(W)^{(1)}$.  An inductive argument using Proposition~\ref{prop: increase sol} concludes that $$R_\eta^n(J) := \overset{n\text{ times}}{\overbrace{R_\eta\circ \dots \circ R_\eta}}(J)$$ is $n$-solvable.     With a careful choice of $R$ and $J$, this construction produces knots which are not $(n.5)$-solvable.  In the following proposition, $\rho_0(J)$ is  $\frac{1}{2\pi}  \int_S \sigma_\omega(J)~d\omega$ where $S$ is the unit circle in $\mathbb{C}$ and $\sigma_\omega$ is the Levine-Tristram signature function of $J$.  

\begin{proposition}[{\cite[Theorem 9.5]{CHL09}}]\label{prop:CHL}
Let $R_\eta$ be a generalized doubling operator with $R$ a slice knot. If $\eta=\eta_1\cup\eta_2$ has two components and  $\Bl_R([\widetilde \eta_{1}],[\widetilde\eta_{2}])\neq 0$
then there exists a constant $D$ so that for any knot with $\Arf(J)=0$ and $|\rho_0(J)|>D$, the knot $R^n_{\eta}(J)$ represents a nontrivial element in $\F_n/\F_{n.5}$.\end{proposition}

\section{Homology boundary links deep in the solvable filtration: The proof of Theorem~\ref{thm: main}}\label{sect: HBL =/= BL for n.5}
We first prove the following proposition.  

\begin{proposition}\label{prop:2ndcommu}
Let $L=L_1\cup L_2$ be a 2-component link and $n\in \frac{1}{2}\mathbb{N}_0$.  
If $L$ is $n$-solve-equivalent to a boundary link and $L_1$ is $n$-solvable, then there is an $n$-solution $Z$ for $L_1$ such that the homotopy class of $L_2$ lies in $\pi_1(Z)^{(2)}$.  
\end{proposition}

When combined with Proposition~\ref{prop: increase sol} this allows us to conclude that if $R_\eta$ is a generalized doubling operator, $R$ is slice and $R\cup \eta$ is $n.5$-solvably concordant to a boundary link, then $J\mapsto R_\eta(J)$ will increase solvability by 2.  We say this more formally as Corollary~\ref{cor: bdry link doubling op}. This will unlock an avenue to a proof of Theorem~\ref{thm: main} by finding a homology boundary link $R\cup \eta$ for which $J\mapsto R_\eta(J)$ does not increase solvability by 2.

\begin{proof}[Proof of Proposition~\ref{prop:2ndcommu}]
As has been our habit, we give the proof when $n\in \N_0$.  When $n
$ is a half integer, the proof is the same.  Let $(W, C=C_1\cup C_2)$ be an $n$-solve-equivalence from $L\subseteq \bdry_+ W$ to a boundary link $J = J_1\cup J_2\subseteq \bdry_- W$.  Then $(W, C_1)$ is an $n$-solve-equivalence from $L_1$ to $J_1$.  As $L_1$ is $n$-solvable, so is $J_1$.  Let $W_1$ be an $n$-solution for $J_1$.  Notice that $\bdry_- W \sm \nu C_1$ is homeomorphic to $S^3 \sm \nu J_1$ by an orentation reversing homeomorphism.  Similarly, in $\bdry W_1\cong M_{J_1}$  we see a copy of $S^3 \sm \nu J_1$.  As in \cite[Proposition 2.4]{Davis:2020-1}, we glue $W-\nu C_1$ to $W_1$ along these copies of $S^3 \sm \nu J_1$.  Similar analysis to that in \cite[Proposition 2.4]{Davis:2020-1}  shows that this manifold, which we shall denote by $Z$, is an $n$-solution for $L_1$. We briefly sketch this analysis.

A Mayer-Veitoris argument reveals that $H_1(Z)\cong \Z$ is generated by the meridian of $L_1$ and that $H_2(Z)\cong H_2(W\sm C_1)\oplus H_2(W_1)\cong H_2(W)\oplus H_2(W_1)$ is generated by the Lagrangians and duals of $W$ and $W_1$.  The functoriality of the derived series implies that these surfaces form $n$-Lagrangians and $n$-duals for $Z$.  

It remains to check that the homotopy class of $L_2$ lies in $\pi_1(Z)^{(2)}$.  The embedded annulus $C_2\subseteq W \sm \nu C_1\subseteq Z$ gives a free homotopy from $L_2$ to $J_2$.  Since $J$ is a boundary link, the homotopy class of $J_2$ lies in $\pi_1(S^3 \sm \nu J_1)^{(2)}$.  
We conclude that the homotopy class of $L_2$ lies in $\pi_1(Z)^{(2)}$.\end{proof}

Given a generalized doubling operator $R_\eta$, note that if $R \cup \eta$ is $n$-solve-equivalent a boundary link and $R$ is $n$-solvable with an $n$-solution $Z$, then Proposition~\ref{prop:2ndcommu} implies that the homotopy class of  each component of $\eta$ lies in $\pi_1(Z)^{(2)}$. Using Proposition~\ref{prop: increase sol}, we immediately get the following:

\begin{corollary}\label{cor: bdry link doubling op}
Let $n\in \N$, $J$ be an $(n-1.5)$-solvable knot, and $R_\eta$ be a generalized doubling operator.  
Suppose that $R\cup\eta$ is $n.5$-solve-equivalent to a boundary link, and $R$ is $n.5$-solvable, then $R_\eta(J)$ is $n.5$-solvable.  
\end{corollary}

Now, we are ready to prove Theorem~\ref{thm: main}
\begin{proof}[Proof of Theorem~\ref{thm: main}]
Let $n$ be a positive integer. We construct a knot $J_n$ so that the link $L(J_n)=L_1 \cup L_2$ of Figure~\ref{fig: highly sol example} satisfies the following:
\begin{enumerate}[font=\upshape]
\item \label{item: main n-sol}$L(J_n)$ is $n$-solvable but not $(n.5)$-solve-equivalent to any boundary link.
\item \label{item: main Milnor}$L(J_n)$ is a homology boundary link with slice components.
\end{enumerate}

For the construction we use the  generalized  doubling operator  $R_{\eta_1, \eta_2}$ of Figure~\ref{fig: doubling}. As is checked in \cite[Example 3.3]{CHL09}, $R_{\eta_1, \eta_2}$ satisfies the hypotheses of Proposition~\ref{prop:CHL}.  Let $D$ be the constant in the proposition and let $J_1$ be a $0$-solvable knot with $|\rho_0(J_1)|>D$.  Such a knot can be obtained by taking a connected sum of a sufficiently large even number of trefoils. Let $J_n = R_{\eta}^{n-1}(J_1)$, then by Proposition~\ref{prop:CHL} and induction, $J_n$ is $(n-1)$-solvable. Hence, again by Proposition~\ref{prop: increase sol}, we have that $L(J_n)$ is $n$-solvable.  

Next, suppose for the sake of contradiction that $L(J_n) = L_1\cup L_2$ is $n.5$-solve-equivalent to a boundary link.   If we think of $(L_1)_{L_2}$ as a doubling operator then $(L_1)_{L_2}(J_n) = R_\eta(J_n) = R_{\eta}^{n}(J_1)$, which is not $n.5$-solvable by Proposition~\ref{prop:CHL}.  On the other hand, since $J_n$ is $(n-1)$-solvable, Corollary~\ref{cor: bdry link doubling op} concludes that $(L_1)_{L_2}(J_n)$ is $n.5$-solvable, a contradiction.

Finally, we verify \pref{item: main Milnor}.  It is straightforward to see that the components of $L(J_n)= L_1\cup L_2$ are slice.  Indeed, $L_2$ is the unknot.  To see a slice disk for $L_1$, notice that on the obvious Seifert surface $S$ for $L_1$ there is a nonseparating unknotted curve with zero on which the Seifert form vanishes.  Use a slice disk for this unknot to modify $S$ by ambient surgery.  What results is a slice disk for $L_1$.  In \cite[Section 3]{Cochran-Orr:1993-1}, it is shown that $L(J_n)$ is a homology boundary link.\end{proof}

\bibliographystyle{alpha}
\bibliography{biblio}

\begin{bibdiv}
\begin{biblist}

\bib{Blanchfield57}{article}{
      author={Blanchfield, Richard~C.},
       title={Intersection theory of manifolds with operators with applications
  to knot theory},
        date={1957},
        ISSN={0003-486X},
     journal={Ann. of Math. (2)},
      volume={65},
       pages={340\ndash 356},
         url={https://doi-org.proxy.uwec.edu/10.2307/1969966},
      review={\MR{85512}},
}

\bib{Ca}{article}{
      author={Casson, Andrew~J.},
       title={Link cobordism and {M}ilnor's invariant},
        date={1975},
        ISSN={0024-6093},
     journal={Bull. London Math. Soc.},
      volume={7},
       pages={39\ndash 40},
      review={\MR{MR0362286 (50 \#14728)}},
}

\bib{CiF}{article}{
      author={Cimasoni, David},
      author={Florens, Vincent},
       title={Generalized {S}eifert surfaces and signatures of colored links},
        date={2008},
        ISSN={0002-9947},
     journal={Trans. Amer. Math. Soc.},
      volume={360},
      number={3},
       pages={1223\ndash 1264},
      review={\MR{MR2357695}},
}

\bib{CHL4}{article}{
      author={Cochran, Tim~D.},
      author={Harvey, Shelly},
      author={Leidy, Constance},
       title={Link concordance and generalized doubling operators},
        date={2008},
     journal={Algebr. Geom. Topol.},
      volume={8},
       pages={1593\ndash 1646},
}

\bib{CHL09}{article}{
      author={Cochran, Tim~D.},
      author={Harvey, Shelly},
      author={Leidy, Constance},
       title={Knot concordance and higher-order {B}lanchfield duality},
        date={2009},
        ISSN={1465-3060},
     journal={Geom. Topol.},
      volume={13},
      number={3},
       pages={1419\ndash 1482},
         url={https://doi-org.proxy.uwec.edu/10.2140/gt.2009.13.1419},
      review={\MR{2496049}},
}

\bib{derivatives}{article}{
      author={Cochran, Tim~D.},
      author={Harvey, Shelly},
      author={Leidy, Constance},
       title={Derivatives of knots and second order signatures},
        date={2010},
     journal={Geometry and Topology},
      volume={10},
      number={10},
       pages={739\ndash 787},
}

\bib{CHL2011}{article}{
      author={Cochran, Tim~D.},
      author={Harvey, Shelly},
      author={Leidy, Constance},
       title={Primary decomposition and the fractal nature of knot
  concordance},
        date={2011},
        ISSN={0025-5831},
     journal={Math. Ann.},
      volume={351},
      number={2},
       pages={443\ndash 508},
         url={https://doi-org.proxy.uwec.edu/10.1007/s00208-010-0604-5},
      review={\MR{2836668}},
}

\bib{Cimasoni04}{article}{
      author={Cimasoni, David},
       title={A geometric construction of the {C}onway potential function},
        date={2004},
        ISSN={0010-2571},
     journal={Comment. Math. Helv.},
      volume={79},
      number={1},
       pages={124\ndash 146},
         url={https://doi-org.proxy.uwec.edu/10.1007/s00014-003-0777-6},
      review={\MR{2031702}},
}

\bib{CocLev91}{article}{
      author={Cochran, Tim~D.},
      author={Levine, Jerome~P.},
       title={Homology boundary links and the {A}ndrews-{C}urtis conjecture},
        date={1991},
        ISSN={0040-9383},
     journal={Topology},
      volume={30},
      number={2},
       pages={231\ndash 239},
         url={https://doi-org.proxy.uwec.edu/10.1016/0040-9383(91)90009-S},
      review={\MR{1098917}},
}

\bib{Cochran-Orr:1990-1}{article}{
      author={Cochran, T.~D.},
      author={Orr, K.~E.},
       title={Not all links are concordant to boundary links},
        date={1990},
        ISSN={0273-0979},
     journal={Bull. Amer. Math. Soc. (N.S.)},
      volume={23},
      number={1},
       pages={99\ndash 106},
         url={https://doi.org/10.1090/S0273-0979-1990-15912-9},
      review={\MR{1031581}},
}

\bib{Cochran-Orr:1993-1}{article}{
      author={Cochran, Tim~D.},
      author={Orr, Kent~E.},
       title={Not all links are concordant to boundary links},
        date={1993},
        ISSN={0003-486X},
     journal={Ann. of Math. (2)},
      volume={138},
      number={3},
       pages={519\ndash 554},
         url={https://doi.org/10.2307/2946555},
      review={\MR{1247992}},
}

\bib{C4}{article}{
      author={Cochran, Tim~D.},
       title={Derivatives of links: {M}ilnor's concordance invariants and
  {M}assey's products},
        date={1990},
        ISSN={0065-9266},
     journal={Mem. Amer. Math. Soc.},
      volume={84},
      number={427},
       pages={x+73},
      review={\MR{MR1042041 (91c:57005)}},
}

\bib{Cochran91}{article}{
      author={Cochran, Tim~D.},
       title={{$k$}-cobordism for links in {$S^3$}},
        date={1991},
        ISSN={0002-9947},
     journal={Trans. Amer. Math. Soc.},
      volume={327},
      number={2},
       pages={641\ndash 654},
         url={https://doi-org.proxy.uwec.edu/10.2307/2001818},
      review={\MR{1055569}},
}

\bib{Conway18}{article}{
      author={Conway, Anthony},
       title={An explicit computation of the {B}lanchfield pairing for
  arbitrary links},
        date={2018},
        ISSN={0008-414X},
     journal={Canad. J. Math.},
      volume={70},
      number={5},
       pages={983\ndash 1007},
         url={https://doi-org.proxy.uwec.edu/10.4153/CJM-2017-051-5},
      review={\MR{3831912}},
}

\bib{Cooper82}{incollection}{
      author={Cooper, D.},
       title={The universal abelian cover of a link},
        date={1982},
   booktitle={Low-dimensional topology ({B}angor, 1979)},
      series={London Math. Soc. Lecture Note Ser.},
      volume={48},
   publisher={Cambridge Univ. Press, Cambridge-New York},
       pages={51\ndash 66},
      review={\MR{662427}},
}

\bib{COT03}{article}{
      author={Cochran, Tim~D.},
      author={Orr, Kent~E.},
      author={Teichner, Peter},
       title={Knot concordance, {W}hitney towers and {$L^2$}-signatures},
        date={2003},
        ISSN={0003-486X},
     journal={Ann. of Math. (2)},
      volume={157},
      number={2},
       pages={433\ndash 519},
         url={https://doi.org/10.4007/annals.2003.157.433},
      review={\MR{1973052}},
}

\bib{CocTei04}{article}{
      author={Cochran, Tim~D.},
      author={Orr, Kent~E.},
      author={Teichner, Peter},
       title={Structure in the classical knot concordance group},
        date={2004},
        ISSN={0010-2571},
     journal={Comment. Math. Helv.},
      volume={79},
      number={1},
       pages={105\ndash 123},
      review={\MR{MR2031301 (2004k:57005)}},
}

\bib{Davis:2020-1}{article}{
      author={Davis, Christopher~W.},
       title={Topological concordance of knots in homology spheres and the
  solvable filtration},
        date={2020},
        ISSN={1753-8416},
     journal={J. Topol.},
      volume={13},
      number={1},
       pages={343\ndash 355},
         url={https://doi.org/10.1112/topo.12126},
      review={\MR{4138741}},
}

\bib{DNOP20}{article}{
      author={Davis, Christopher~W.},
      author={Nagel, Matthias},
      author={Orson, Patrick},
      author={Powell, Mark},
       title={Surface systems and triple linking numbers},
        date={2020},
        ISSN={0022-2518},
     journal={Indiana Univ. Math. J.},
      volume={69},
      number={7},
       pages={2505\ndash 2547},
         url={https://doi.org/10.1512/iumj.2020.69.8081},
      review={\MR{4195611}},
}

\bib{Donaldson}{article}{
      author={Donaldson, Simon~K.},
       title={An application of gauge theory to four-dimensional topology},
        date={1983},
        ISSN={0022-040X},
     journal={J. Differential Geom.},
      volume={18},
      number={2},
       pages={279\ndash 315},
         url={http://projecteuclid.org/euclid.jdg/1214437665},
      review={\MR{710056}},
}

\bib{Davis-Park-2021}{article}{
      author={Davis, Christopher~W.},
      author={Park, JungHwan},
       title={Concordance to links with an unknotted component},
        date={2021},
        ISSN={0305-0041},
     journal={Math. Proc. Cambridge Philos. Soc.},
      volume={170},
      number={1},
       pages={155\ndash 160},
         url={https://doi-org.proxy.uwec.edu/10.1017/S0305004119000367},
      review={\MR{4201550}},
}

\bib{FrPo17}{article}{
      author={Friedl, Stefan},
      author={Powell, Mark},
       title={A calculation of {B}lanchfield pairings of 3-manifolds and
  knots},
        date={2017},
        ISSN={1609-3321},
     journal={Mosc. Math. J.},
      volume={17},
      number={1},
       pages={59\ndash 77},
  url={https://doi-org.proxy.uwec.edu/10.17323/1609-4514-2017-17-1-59-77},
      review={\MR{3634521}},
}

\bib{Freedman1}{article}{
      author={Freedman, Michael~H.},
       title={The topology of four-dimensional manifolds},
        date={1982},
        ISSN={0022-040X},
     journal={J. Differential Geometry},
      volume={17},
      number={3},
       pages={357\ndash 453},
         url={http://projecteuclid.org/euclid.jdg/1214437136},
      review={\MR{679066}},
}

\bib{KirbyCalculus}{book}{
      author={Gompf, Robert~E.},
      author={Stipsicz, Andr\'as~I.},
       title={{$4$}-manifolds and {K}irby calculus},
      series={Graduate Studies in Mathematics},
   publisher={American Mathematical Society, Providence, RI},
        date={1999},
      volume={20},
        ISBN={0-8218-0994-6},
         url={https://doi.org/10.1090/gsm/020},
      review={\MR{1707327}},
}

\bib{HL2}{article}{
      author={Habegger, Nathan},
      author={Lin, Xiao-Song},
       title={On link concordance and {M}ilnor's {$\overline {\mu}$}
  invariants},
        date={1998},
        ISSN={0024-6093},
     journal={Bull. London Math. Soc.},
      volume={30},
      number={4},
       pages={419\ndash 428},
      review={\MR{MR1620841 (2000e:57023)}},
}

\bib{Kearton75}{article}{
      author={Kearton, Cherry},
       title={{B}lanchfield duality and simple knots},
        date={1975},
     journal={Transaction of the Americal Mathematical Society},
      volume={202},
}

\bib{KiHyoungKo1}{article}{
      author={Ko, Ki~Hyoung},
       title={Seifert matrices and boundary link cobordisms},
        date={1987},
        ISSN={0002-9947},
     journal={Trans. Amer. Math. Soc.},
      volume={299},
      number={2},
       pages={657\ndash 681},
         url={https://doi.org/10.2307/2000519},
      review={\MR{869227}},
}

\bib{KiHyoungKo2}{article}{
      author={Ko, Ki~Hyoung},
       title={Seifert matrices and boundary link cobordisms},
        date={1987},
        ISSN={0002-9947},
     journal={Trans. Amer. Math. Soc.},
      volume={299},
      number={2},
       pages={657\ndash 681},
         url={https://doi.org/10.2307/2000519},
      review={\MR{869227}},
}

\bib{Livingston:1990-1}{article}{
      author={Livingston, Charles},
       title={Links not concordant to boundary links},
        date={1990},
        ISSN={0002-9939},
     journal={Proc. Amer. Math. Soc.},
      volume={110},
      number={4},
       pages={1129\ndash 1131},
         url={https://doi.org/10.2307/2047766},
      review={\MR{1031670}},
}

\bib{MartinThesis}{book}{
      author={Martin, Taylor~E.},
       title={Lower order solvability of links},
   publisher={ProQuest LLC, Ann Arbor, MI},
        date={2013},
        ISBN={978-1303-62072-0},
  url={http://gateway.proquest.com.proxy.uwec.edu/openurl?url_ver=Z39.88-2004&rft_val_fmt=info:ofi/fmt:kev:mtx:dissertation&res_dat=xri:pqm&rft_dat=xri:pqdiss:3577550},
        note={Thesis (Ph.D.)--Rice University},
      review={\MR{3211444}},
}

\bib{Martin22}{article}{
      author={Martin, Taylor~E.},
       title={Classification of links up to 0-solvability},
        date={2022},
        ISSN={0166-8641},
     journal={Topology and its Applications},
      volume={317},
       pages={108158},
  url={https://www.sciencedirect.com/science/article/pii/S0166864122001602},
}

\bib{MelMel03}{article}{
      author={Mellor, Blake},
      author={Melvin, Paul},
       title={A geometric interpretation of {M}ilnor's triple linking numbers},
        date={2003},
        ISSN={1472-2747},
     journal={Algebr. Geom. Topol.},
      volume={3},
       pages={557\ndash 568},
         url={https://doi-org.proxy.uwec.edu/10.2140/agt.2003.3.557},
      review={\MR{1997329}},
}

\bib{NaSt03}{article}{
      author={Naik, Swatee},
      author={Stanford, Theodore},
       title={A move on diagrams that generates {$S$}-equivalence of knots},
        date={2003},
        ISSN={0218-2165},
     journal={J. Knot Theory Ramifications},
      volume={12},
      number={5},
       pages={717\ndash 724},
         url={https://doi-org.proxy.uwec.edu/10.1142/S0218216503002639},
      review={\MR{1999640}},
}

\bib{Otto14}{article}{
      author={Otto, Carolyn},
       title={The {$(n)$}-solvable filtration of link concordance and
  {M}ilnor's invariants},
        date={2014},
        ISSN={1472-2747},
     journal={Algebr. Geom. Topol.},
      volume={14},
      number={5},
       pages={2627\ndash 2654},
         url={https://doi-org.proxy.uwec.edu/10.2140/agt.2014.14.2627},
      review={\MR{3276843}},
}

\bib{Rolfsen}{book}{
      author={Rolfsen, Dale},
       title={Knots and links},
      series={Mathematics Lecture Series},
   publisher={Publish or Perish Inc.},
     address={Houston, TX},
        date={1990},
      volume={7},
        ISBN={0-914098-16-0},
        note={Corrected reprint of the 1976 original},
      review={\MR{MR1277811 (95c:57018)}},
}

\bib{Sheiham03}{article}{
      author={Sheiham, Desmond},
       title={Invariants of boundary link cobordism},
        date={2003},
        ISSN={0065-9266},
     journal={Mem. Amer. Math. Soc.},
      volume={165},
      number={784},
       pages={x+110},
         url={https://doi-org.proxy.uwec.edu/10.1090/memo/0784},
      review={\MR{1997849}},
}

\end{biblist}
\end{bibdiv}
\end{document}